\documentclass[usenames,dvipsnames]{article}
\usepackage[utf8]{inputenc}
\usepackage{hyperref}
\usepackage{listings, color}
\usepackage{tabularx}
\usepackage{booktabs} 
\usepackage{bm}

\usepackage{tabularx}
\usepackage{caption}
\usepackage{wrapfig}
\usepackage{adjustbox}
\usepackage{multirow}
\usepackage{mwe}
\usepackage[]{graphicx}
\usepackage{caption}
\usepackage{subcaption}
\usepackage{tikz}

\usepackage[T1]{fontenc}    % use 8-bit T1 fonts
\usepackage{hyperref}
\hypersetup{
    colorlinks=true,
    linkcolor=blue,
    filecolor=magenta,      
    urlcolor=cyan,
}
\usepackage{url}            % simple URL typesetting
\usepackage{booktabs}       % professional-quality tables
\usepackage{amsfonts}       % blackboard math symbols
\usepackage{nicefrac}       % compact symbols for 1/2, etc.
\usepackage{microtype}      %

\usepackage{graphicx}
\usepackage{booktabs} % for professional tables

\usepackage{framed}
\usepackage{amssymb}
\usepackage{amsfonts}
\usepackage{mathrsfs}
\usepackage{changepage}

\usepackage{mathtools}
\usepackage{array}
\usepackage{amsthm}
\usepackage{verbatim} 
\usepackage{enumerate}
\usepackage{bbm}
\usepackage{commath}
\usepackage{wrapfig}
\usepackage{amsbsy}
\usepackage{float}
\usepackage{amsmath}
\usepackage{algorithm}
\usepackage[numbers]{natbib}
\usepackage{tabularx}
\usepackage{listings}
\usepackage{enumitem}

\usepackage{ssArxiv}

\newtheorem{prop}{Proposition}

\newtheorem{lemma}[prop]{Lemma}

\theoremstyle{definition}

\title{Introducing Discrepancy Values of Matrices with Application to Bounding Norms of Commutators\thanks{This research is partially supported by \textsc{Nsf-Career} grant (IIS-1846088), \textsc{Nsf-Tripods+X} grant (DMS-1839258), and the EECS MathWorks Engineering Fellowship.}}

\author{\name{Pourya Habib Zadeh}\email{pourya@mit.edu}\\
\name{Suvrit Sra}\email{suvrit@mit.edu}\\\addr{Massachusetts Institute of Technology, Cambridge, MA, USA}}

\numberwithin{equation}{section}

\begin{document}

%\leftskip=100pt\rightskip=100pt
\maketitle
%\leftskip=0pt\rightskip=0pt

\begin{abstract} 
 We introduce \emph{discrepancy values}, quantities inspired by the notion of the spectral spread of Hermitian matrices.  We define them as the discrepancy between two consecutive Ky-Fan-like seminorms. As a result, discrepancy values share many properties with singular values and eigenvalues, yet are substantially different to merit their own study. We describe key properties of discrepancy values, and establish several tools such as representation theorems,  majorization inequalities, convex formulations, etc., for working with them. As an important application, we illustrate the role of discrepancy values in deriving tight bounds on the norms of commutators. 
\end{abstract}

\vspace{-4pt}
\section{Introduction}
Eigenvalues and singular values are fundamental objects associated with linear transformations; their importance across all of science and engineering hardly needs an introduction~\cite{horn2012matrix}. A lesser-known, but important quantity based on eigenvalues is \emph{spectral spread}, defined for an $n \times n$ Hermitian matrix $\ma$ as
\begin{equation}
  \label{eq:1}
  \text{spr}^+(\ma) := (\lambda_i(\ma)-\lambda_{n-i+1}(\ma))_{i = 1, \dots, \floor*{n/2}},
\end{equation}
where the eigenvalues $\lambda_i$ are ordered decreasingly. The spectral spread was introduced in~\cite{knyazev2010rayleigh} as a measure of the dispersion of the eigenvalues of Hermitian matrices; it is a vector extension of the better known \emph{spread} of Hermitian operators~\cite{bhatia2008commutators}, which is equal to the greatest eigenvalue minus the smallest eigenvalue of the Hermitian matrix. %$\text{spd}(\ma):=\lambda_{\max}(\ma)-\lambda_{\min}(\ma)$. % There exist many useful objects that quantify properties of linear transformations, each of which extracts certain aspects of a linear map. To name two such notions, we can point out  eigenvalues and singular values of linear operators. Eigenvectors are vectors whose direction is invariant under the linear transformation, and eigenvalues are the corresponding scaling of those vectors under the transformation.
%There are many possible definitions for singular values; however, they can be best described geometrically. Consider a unit sphere, a (full rank) linear map would transform the sphere into an ellipsoid. 
%The length of the semi-axes of this ellipsoid is the non-zero singular values, and the directions of the semi-axes are the singular vectors. The eigenvalue and singular values are extremely useful tools in linear algebra and are prevalent in engineering and science~\cite{horn2012matrix}.
%The spectral spread of Hermitian matrices is another fundamental example. It was first introduced in This measure
Spectral spread derives its importance from a class of useful inequalities it satisfies, see e.g.,~\cite{massey2020absolute, massey2021norm, massey2021spectral}. But its definition is limited to Hermitian matrices and for general square (complex) matrices there is no clear analog.

One could define the spectral spread for the general matrix $\ma$ as an extension of the notion of the spread of square matrices~\cite{mirsky1956spread}: $\text{spd}(\ma):= \max_{i,j} |\lambda_{i}(\ma)-\lambda_j(\ma)|$. After finding the spread, we repeatedly remove the two indices that maximize the spread, find the spread of the remaining eigenvalues, and insert it in a vector. When $\ma$ is a Hermitian matrix, the obtained vector is equivalent to~\ref{eq:1}. However, for non-Hermitian matrices, this quantity does not enjoy many nice properties that the spectral spread of a Hermitian function possesses. Therefore, we need to find an alternative analog that satisfies all the nice properties of the spectral spread.

%It is done consecutively. Since $\max_{i,j} |\lambda_i(\ma) − \lambda_j(\ma)|$ is attained for some $i, j$, the next step is to remove these two eigenvalues of $\ma$ and consider the spread of the remaining $n − 2$ eigenvalues. This process could go on.

We motivate, propose, and analyze an analog that \emph{is} defined for general square matrices: namely, \emph{discrepancy values}, which we define as the discrepancy between consecutive Ky-Fan-like seminorms (see Def.~\ref{Def:radii}). We will show that discrepancy values can be essentially thought of as two copies of spectral spread. Alternatively, they can also  be thought of as a (vector-valued) cost of approximating linear operators via scalar multiples of the identity, a topic of some importance---see e.g.,~\cite{bhatia1999orthogonality,stampfli1970norm, Apostol84}.

%After defining discrepancy values, we will prove several important theorems about them, including the connection that they have with singular values, Ky-Fan norms, and the principal angles between subspaces. We will also discover the invariances they satisfy, as well as several majorization inequalities. 
%We will see that many important results about the spectral spread also hold for discrepancy values. We then employ the discrepancy values and the set of tools that we develop to prove a collection of other results, of which the most important consequence is a tight majorization inequality satisfied by the singular values of the commutator of two square matrices.

\subsection{Outline of paper and main contributions}
We first define discrepancy values (Definition~\ref{Def:radii}) in Section~\ref{sec: def}; next, we elucidate their connections with Ky-Fan norms (Theorems~\ref{max_delta},~\ref{lemma_min}), singular values (Proposition~\ref{singeqdisc}, and inequality~\eqref{eq:6}), and principal angles between subspaces (Theorem~\ref{Thm:principalAngle}) in Section~\ref{sec: properties}. Through several results we explore invariances and majorization inequalities satisfied by discrepancy values while underscoring their analogy with spectral spread (see e.g., Corollary~\ref{corr:block}).
% =======
% After defining discrepancy values (Definition~\ref{Def:radii}) in Section~\ref{sec: def}, we elucidate their connections with Ky-Fan norms (Theorems~\ref{max_delta},~\ref{lemma_min}), singular values (Proposition~\ref{singeqdisc}, and Inequality~\eqref{eq:6}), and principal angles between subspaces~\cite{qiu2005unitarily} (see Theorem~ \ref{Thm:principalAngle}) in Section~\ref{sec: properties}. Therein, through several theorems and supporting results we explore invariances and majorization inequalities satisfied by discrepancy values while underscoring their analogy with spectral spread (see e.g., Corollary~\ref{corr:block}).

In Section~\ref{sec: application} we apply our results on discrepancy values to study commutators, a subject that has received extensive   interest~\cite{kittaneh2007inequalities,bottcher2005big,hirzallah2009commutator,lu2017commutator,wenzel2010impressions,wenzel2015strange,stampfli1970norm}. Specifically, we obtain tight bounds on norms of commutators using discrepancy values (Corollaries~\ref{coroll:zero1},~\ref{generalMajIneq}), and determine when two Hermitian matrices with fixed eigenvalues are maximally non-commutative (Theorem~\ref{thm1}). %We will see that many important results about the spectral spread also hold for discrepancy values. We then employ the discrepancy values and the set of tools that we develop to prove a collection of other results, of which the most important consequence is a tight majorization inequality satisfied by the singular values of the commutator of two square matrices.
%Commutators are pivotal objects in subjects such as Lie algebra and group theory. Algebraically, the commutator of two matrices $\ma$ and $\mb$ is defined using the formula $[\ma,\mb]=\ma\mb-\mb\ma$. It can be viewed as a quantity that measures how much two operators are (non-)commuting. 
%Finding bounds on the norms of such quantities has been investigated extensively in the literature~\cite{kittaneh2007inequalities,bottcher2005big,hirzallah2009commutator,lu2017commutator,wenzel2010impressions,wenzel2015strange,stampfli1970norm}.
%As we will show in the paper, we can obtain sharp bounds on the norm of commutators using discrepancy values, and find when two Hermitian matrices with fixed eigenvalues are maximally non-commutative.
Thereafter, in Section~\ref{sec: calculation}, we propose approaches for calculating discrepancy values via semi-definite programming (see e.g., Equation~\eqref{SDP_general}). %Finally, we discuss that our result can be easily be extended to the case of compact operators on Hilbert space.
We conclude with a discussion about the extension of our results to compact operators on Hilbert spaces, and with a conjecture that posits a stronger majorization inequality involving commutators and discrepancy values (Conjecture~\ref{conj:squareM}).

%\subsection{Contributions} In light of the above, we outline our main contributions.
% \begin{itemize}
% \item Introducing the notion of discrepancy values as the generalization of spectral spread of Hermitian matrices to any square matrices .
% \item Proving several important results about this object such as the equivalence of two representation of the discrepancy values and many majorization bounds (Section~\ref{sec: properties}).
% \item Using this new notion to come up with tight bounds for the commutators and consequently finding when two Hermitian matrices with fixed eigenvalues become maximally non-commutative. ().
% \item Providing methods to efficiently compute the discrepancy values of matrices (Section~\ref{sec: calculation}).
% \end{itemize}

%%% Local Variables:
%%% mode: latex
%%% TeX-master: "main"
%%% End:

\section{Notation and Preliminaries}\label{sec: preli}
Bold letters denote vectors and matrices, both of which are assumed to be complex unless specified otherwise. The vectors $\vsigma(\ma)$ and $\vlambda(\ma)$ respectively denote singular values and eigenvalues of $\ma$. % $m \times n$ matrix $\ma$ with complex-valued entries, denoted by $\ma \in \mm_{m,n}(\mathbb{C})$.
We assume that the singular values of any matrix, and eigenvalues of Hermitian matrices are ordered decreasingly. Let $\mi_n$ be the $n \times n$ identity; and $\mj_n$ the exchange matrix that has  ones on its antidiagonal and zeros elsewhere; $\mathbf{1}_k$ is a vector with $k$ ones, followed by $n-k$ zeros. For $\vx \in \mathbb{R}^n$, $\vx^{\downarrow}$ is the vector with coordinates arranged  nonincreasingly. % ; i.e., $x^{\downarrow}_{i}\geq x^{\downarrow}_{j}$ for $1\leq i \leq j \leq n$. 
The \emph{Schatten norm} of an $n\times n$ matrix $\ma$ is defined by
$\|\ma\|_{p}:=(\sum_{i=1}^{n}\sigma^{p}_i(\ma))^{1/p}$, for $p \geq 1$, while its \emph{Ky-Fan} norms are defined as
$\|\ma\|_{(k)}:=\sum_{i=1}^k\sigma_i(\ma)$, for $k=1,\dots, n$. Ky-Fan norms enjoy the following maximal formulation:
\begin{equation}
  \|\ma\|_{(k)} = \max \Big\{ \Big| \nlsum_{i=1}^k \langle \ma\vx_i,\vy_i \rangle \Big|\ \ :\ \ \{\vx_j\}_{j=1}^k \; \text{o.n.},\ \{\vy_j\}_{j=1}^k \; \text{o.n.}\Big\},
   \label{eq:KF}
\end{equation}
where `o.n.' means that the set of vectors $\{\vx_j\}_{j=1}^k$ is orthonormal.

An $n\times r$ matrix $\ma$ is called an isometry if $\ma^*\ma=\mi_{r}$ for some $r \leq n$, and a partial isometry of order $k \leq n=r$ if $\ma = \mb \mc^*$ where $\mb \in \C^{n\times k}$ and $\mc \in \C^{n\times k}$ are both isometries. Recall that a norm $\|\cdot\|$ is \emph{unitarily invariant} if $\|\ma\|= \|\mU\ma\mv\|$ for all unitary matrices $\mU,\mv$. In the sequel, we use $\vertiii{\cdot}$ to denote such norms. Both Schatten and Ky-Fan norms are unitarily invariant.

Finally, recall that for $\vx,\vy \in \mathbb{R}^n$, we say that $\vx$ is \emph{weakly majorized} by $\vy$, denoted $\vx \prec_w \vy$, if the following set of inequalities hold:
\begin{equation*}
    \nlsum_{i=1}^k x_i^{\downarrow} \leq \nlsum_{i=1}^k y_i^{\downarrow}, \quad\text{for}\ k=1,\ldots,n.
\end{equation*}
Moreover, $\vy$ majorizes $\vx$, denoted $\vx \prec \vy$, when the equality $\sum_{i=1}^n x_i= \sum_{i=1}^n y_i$ also holds. Many fundamental majorization inequalities are known for eigenvalues and singular values; for instance for two square matrices $\ma$ and $\mb$ it is known that
\begin{align*}
  %\label{eq:2}
  \vsigma(\ma+\mb) & \prec_w \vsigma(\ma)+\vsigma(\mb),\quad\text{and}\quad
  %\label{eq:3}
  \vsigma(\ma\mb) \prec_w \vsigma(\ma)\vsigma(\mb),
\end{align*}
where $\vsigma(\ma)\vsigma(\mb)$ is the entrywise product of the vectors of the singular values. We refer the readers to~\cite{ando1994majorizations, bellman1997introduction, bhatia2013matrix, hiai2014introduction, horn2012matrix, zhan2013matrix}, for a deeper study of these topics. %more information about the majorization inequalities, unitarily invariant norms, and variational characterization of Ky-Fan norms.

%%% Local Variables:
%%% mode: latex
%%% TeX-master: "main"
%%% End:

\section{Definition of discrepancy values}\label{sec: def}
Now we are ready to introduce discrepancy values. We need to first introduce a new class of seminorms that are obtained via an innocuous-looking modification to the Ky-Fan norm~\eqref{eq:KF}: namely, an additional partial-orthogonality constraint $\sum_j\langle \vx_j, \vy_j\rangle=0$.  Why we consider the constraint  in~\eqref{main_def}, and what ramifications it has will unravel in the subsequent sections.
\begin{definition}
\label{Def:radii}
Let $\ma \in \C^{n\times n}$; we define its $k\textsuperscript{th}$ \emph{discrepancy seminorm} as
\begin{equation}
\label{main_def}
   \|\ma\|_{(k)}^{\delta} := \max_{\substack{\{\vx_j\}_{j=1}^k \; \text{o.n.} \\\{\vy_j\}_{j=1}^k \; \text{o.n.}}} \bigg| \sum_{i=1}^k \langle \ma\vx_i,\vy_i \rangle\bigg|,\quad\text{s.t.}\ \nlsum_{j=1}^k\langle \vx_j, \vy_j \rangle=0,\ \text{for } 1\le k \le n.
\end{equation}
Further, let $\|\ma\|_{(0)}^{\delta}=0$. Using~\eqref{main_def} we define the $k\textsuperscript{th}$ \emph{discrepancy value} of $\ma$ as
\begin{equation}
  \label{eq:4}
    \delta_k(\ma) := \|\ma\|_{(k)}^{\delta} - \|\ma\|_{(k-1)}^{\delta}.
\end{equation}
\end{definition}
Clearly, $\|\ma\|^{\delta}_{(k)} \leq \|\ma\|_{(k)}$; equivalently, writing $\vdelta(\ma) := (\delta_1(\ma), \dots, \delta_n(\ma))$ (decreasingly sorted), we obtain $\vdelta(\ma)\prec_w\sigma(\ma)$. Moreover, $\sigma_1(\ma) \geq \delta_1(\ma)\geq \delta_2(\ma) \geq \dots \geq \delta_n(\ma)\geq 0$. But unlike singular values, there may not exist two sets of orthonormal vectors $\{\vx_j\}$ and $\{\vy_j\}$ such that $\vx_i \bot \vy_i$ and  $\delta_i(\ma) = |\langle \ma\vx_i,\vy_i \rangle|$. Nevertheless, we will call the vectors $\vx_i$ and $\vy_i$ in \eqref{main_def}, the left and right \emph{discrepancy vectors}. 

%From the definition~\eqref{eq:4} our choice of name is evident, since the discrepancy value is defined as the discrepancy (difference) between two consecutive seminorms. Equivalently, one could define $\delta_i(\ma)$ first, and then define $\|\ma\|_{(k)}^{\delta} = \sum_{i=1}^k \delta_i(\ma)$. We will also write $\vdelta(\ma) = (\delta_1(\ma), \dots, \delta_n(\ma))$; ordered in a decreasing manner. In analogy with singular vectors, 
%Note that the optimization problem in formula~\ref{main_def} is always feasible and ; hence, $\|\ma\|^{\delta}_{(k)}$ is well-defined.
%We can see that $\sigma_1(\ma) \geq \delta_1(\ma)\geq \delta_2(\ma) \geq \dots \geq \delta_n(\ma)\geq 0$, hence $\delta_i$ is always finite and unique, thus well-defined.
%In fact, we can define the singular values as $\sigma_k(\ma) = \|\ma\|_{(k)} - \|\ma\|_{(k-1)}$, and use the maximal representation of the Ky-Fan norm~\ref{eq:KF} (assuming that $\|\ma\|_{(0)}=0$).

Before proceeding to our main results, let us note down more compact forms for the discrepancy seminorms and Ky-Fan norms.

\begin{proposition}
  \label{prop:maxrep}
For an arbitrary $n \times n$ matrix $\ma$, we have the following:
\begin{align}
    \|\ma\|_{(k)} &= \max_{\substack{ \mm \in \mathcal{P}_{k}(n)}} \R \trace(\ma \mm), \label{compact_Kynorm} \\ 
    \|\ma\|_{(k)}^{\delta} &= \max_{\substack{ \mm \in \mathcal{P}^0_{k}(n)}} \R \trace(\ma \mm), \label{compact_Discnorm}
\end{align}
where $\mathcal{P}_{k}(n):=\{\mx \in \C^{n\times n} : \mx = \mv \mU^*, \mU^*\mU = \mi_k, \mv^*\mv = \mi_k \}$ is the set of $n \times n$ partial isometries of rank $k$, and $\mathcal{P}^0_{k}(n):=\{\mx \in \C^{n\times n} : \mx = \mv \mU^*, \mU^*\mU = \mi_k, \mv^*\mv = \mi_k, \trace(\mx)=0 \}$ is the set of $n \times n$ traceless partial isometries of rank $k$.
\end{proposition}
\begin{proof}
Given a positive integer $k\leq n$, the set
$$\bigg\{\sum_{j=1}^k \langle \ma\vx_i,\vy_i \rangle \ \ :\ \   \{\vx_j\}_{j=1}^k \; \text{o.n.}, \{\vy_j\}_{j=1}^k \; \text{o.n.}, \sum_{j=1}^k \langle \vx_j,\vy_j \rangle = 0\bigg\}$$ forms a circle in the complex plane whose center is located at the origin. To observe this fact, note that we can multiply each vector in the set $\{\vy_j\}_{j=1}^k$ with $e^{i \theta}$ and do not violate the conditions of the set while rotating the value of $\,\sum_{j=1}^k \langle \ma\vx_i,\vy_i \rangle$ arbitrarily in the complex plane.
Therefore, we can replace the modulus in the definitions of $\|\cdot\|_{(k)}$ and $\|\cdot\|^{\delta}_{(k)}$ with the real part of the complex number. 
Thus,
\begin{equation*}
\begin{split}
   \|\ma\|_{(k)} &= \max_{\substack{\|\vx_i\|=\|\vy_i\|=1 \\ \vx_1 \bot \dots \bot \vx_k \\ \vy_1 \bot \dots \bot \vy_k}} \R \sum_{i=1}^k \langle \ma\vx_i,\vy_i \rangle = \max_{\substack{\mU^*\mU = \mi_k \\ \mv^*\mv = \mi_k}} \R \trace(\ma \mv \mU^*), \\
   \|\ma\|_{(k)}^{\delta} &= \max_{\substack{\|\vx_i\|=\|\vy_i\|=1 \\ \vx_1 \bot \dots \bot \vx_k \\ \vy_1 \bot \dots \bot \vy_k \\ \sum_{i=1}^k\langle \vx_i, \vy_i \rangle=0}} \R \sum_{i=1}^k \langle \ma\vx_i,\vy_i \rangle = \max_{\substack{\mU^*\mU = \mi_k \\ \mv^*\mv = \mi_k \\ \trace(\mU^*\mv)=0}} \R \trace(\ma \mv \mU^*).
\end{split}
\end{equation*}
Setting $\mm = \mv\mU^*$ we obtain the desired result.
\end{proof}

From the definition of discrepancy values, the following invariances are immediate.

\begin{proposition}[invariances of discrepancy]
Let $\ma \in \C^{n\times n}$. Discrepancy values display the following invariances:
\begin{itemize}
    \item(Unitary conjugation): $\vdelta(\mU\ma\mU^*) = \vdelta(\ma)$, for all unitary matrix $\mU$.
    \item(Conjugate transpose): $\vdelta(\ma^*) = \vdelta(\ma)$.
    \item(Phase): $\vdelta(e^{i\theta}\ma) = \vdelta(\ma)$, for all $\theta \in [0,2\pi)$.
    \item(Shift): $\vdelta(\ma -\alpha\mi_n) = \vdelta(\ma)$, for all $\alpha \in \mathbb{C}$.
\end{itemize}
\end{proposition}

%%% Local Variables:
%%% mode: latex
%%% TeX-master: "main"
%%% End:

\section{Main results}\label{sec: properties}
In this section, we study fundamental properties of discrepancy values that should be of broader interest. We first take a closer look at discrepancy seminorms and their relation to Ky-Fan norms in Section~\ref{sec:discVsKF}. Then, in Section~\ref{sec:discHerm} we uncover the precise relation between discrepancy values and spectral spread of Hermitian matrices. In Section~\ref{sec:subclasses}, we discuss special classes of matrices that have zero-one discrepancy values, and in analogy with singular values, we also advance the notion of ``discrepancy rank.'' Basic majorization inequalities and a generalization to direct sums of matrices are covered in Sections~\ref{sec:basicMajor} and \ref{sec:directsum}, respectively.

\subsection{Discrepancy seminorms and Ky-Fan norms}
\label{sec:discVsKF}
We develop a more thorough connection between discrepancy seminorms and Ky-Fan norms. We will need the following simple observation. %Before mentioning the theorem, we need the following lemma.
\begin{lemma}
\label{lem:ortheq}
For any two $n \times m$ isometries $\mU, \mv$, where $n \geq m$, there exists a $n \times n$ unitary matrix $\mq$ such that $\mq\mU=\mv$.
\end{lemma}
\begin{proof}
Immediate upon noting that we can write
$\mU_{n\times m}=\mr_{n \times n}{\scriptsize\begin{bmatrix}\mi_{m}\\ \mathbf{0}_{(n-m)\times m}\end{bmatrix}}$, for an appropriate unitary matrix $\mr$.
\end{proof}

The first main result of this section is Theorem~\ref{max_delta} that shows how Ky-Fan norms can be realized via maximal discrepancy seminorms of unitarily transformed matrices.
\begin{theorem}
\label{max_delta}
For the square matrix $\ma$, we have
\begin{equation}
    \label{eq:equiv}
    \|\ma\|_{(k)} = \max_{\mq \in \mathbf{U}(n)} \|\ma\mq\|^{\delta}_{(k)},
\end{equation}
where $\mathbf{U}(n)$ denotes the set of $n \times n$ unitary matrices.
\end{theorem}

\begin{proof}
We want to show that $\max_{\mq \in \mathbf{U}(n)} \sum_{i=1}^k \delta_i(\ma\mq) = \|\ma\|_{(k)}$.
Trivially, $\vdelta(\ma \mq) \prec_w \vsigma(\ma)$ for any unitary $\mq$; i.e., $\max_{\mq \in \mathbf{U}(n)} \sum_{i=1}^k \delta_i(\ma\mq) \leq \|\ma\|_{(k)}$. Using the variational formula~(\ref{compact_Discnorm}) of discrepancy seminorms and noting that $\vdelta(\mr^* \ma \mr)=\vdelta(\ma)$ for any unitary matrix $\mr$, we see that the RHS of~\eqref{eq:equiv} equals
\begin{equation}
    \max_{\mq,\mr \in \mathbf{U}(n)} \max_{\substack{\mU^*\mU = \mi_k \\ \mv^*\mv = \mi_k \\ \trace(\mU^*\mv)=0}} \R \trace\big(\ma (\mq \mv) (\mr\mU)^*\big).
\end{equation}
Also, by the variational formula for Ky-Fan norms, the LHS of~\eqref{eq:equiv} is
\begin{equation}
\label{eq:5}
\max_{\substack{\mm^*\mm = \mi_k \\ \mn^*\mn = \mi_k}} \R \trace(\ma \mn \mm^*).
\end{equation}
Assume that the optimum of~\eqref{eq:5} occurs at $\hat{\mm}$ and $\hat{\mn}$. By Lemma~\ref{lem:ortheq}, for isometries $\mU$ and $\mv$ there exist $\mq,\mr \in \mathbf{U}(n)$ such that $\hat{\mn}=\mq \mv$ and $\hat{\mm}=\mr\mU$. Hence, $\max_{\mq \in \mathbf{U}(n)} \sum_{i=1}^k \delta_i(\ma\mq) \geq \|\ma\|_{(k)}$, which concludes the proof.
\end{proof}

Proposition~\ref{prop:maxrep} provides a maximal representation for discrepancy seminorms. Now we want to provide a minimal representation. This dual representation reveals another aspect of the connection between discrepancy seminorms and Ky-Fan norms. Before mentioning the relationship, let us state two results.
\begin{lemma}
  \label{lem:convhull}
The set $\mathcal{P}_{k}(n)$ is the set of extreme points of the compact convex set $ \mathrm{Conv}(\mathcal{P}_{k}(n)):=\{\mx \in \C^{n\times n} : \|\mx\|_{(1)} \leq 1, \|\mx\|_{(n)} \leq k\}$. Similarly, the set $\mathcal{P}^0_{k}(n)$ is the set of extreme points of the compact convex set $\mathrm{Conv}(\mathcal{P}^0_{k}(n)) :=\{\mx \in \C^{n\times n} : \|\mx\|_{(1)} \leq 1, \|\mx\|_{(n)} \leq k, \trace(\mx)=0\}$.
\end{lemma}
Lemma~\ref{lem:convhull}, together with representations~\eqref{compact_Kynorm} and~\eqref{compact_Discnorm} implies that
\begin{align}
    \|\ma\|_{(k)} & = \max_{\substack{ \mm \in \text{Conv}(\mathcal{P}_{k}(n))}} \R \trace(\ma \mm) \label{cmp_Kynorm} \\
    \|\ma\|_{(k)}^{\delta} & = \max_{\substack{ \mm \in \text{Conv}(\mathcal{P}^0_{k}(n))}} \R \trace(\ma \mm), \label{cmp_Discnorm}
\end{align}
since the cost is linear and we replaced the sets $\mathcal{P}_{k}(n)$, $\mathcal{P}^0_{k}(n)$ with their convex hulls.

\begin{theorem}[Sion's minimax theorem]
Let $X$ be a compact convex subset of a linear topological space and $Y$ a convex subset of a linear topological space. If $f$ is a real-valued function on $X \times Y$ such that $f(\vx, \cdot)$ is upper semicontinuous and quasi-concave on $Y$ for any $\vx \in X$ and $f(\cdot, \vy)$ is lower semicontinuous and quasi-convex on $X$ for any $\vy \in Y$, then we have
\begin{equation}
    \min_{\vx \in X} \sup_{\vy \in Y} f(\vx,\vy) = \sup_{\vy \in Y} \min_{\vx \in X} f(\vx,\vy).
\end{equation}
\end{theorem}
The announced minimal representation is noted in Theorem~\ref{lemma_min}.
\begin{theorem}
\label{lemma_min}
Let $\ma \in \C^{n\times n}$. For $1 \le k \le n$, we have the variational formula
\begin{equation}
\label{eq:minchar}
    \|\ma\|^{\delta}_{(k)} = \min_{\alpha \in \mathbb{C}} \|\ma-\alpha \mi_n\|_{(k)}.
\end{equation}
\end{theorem}
\begin{proof}
Using representation~\eqref{cmp_Discnorm} for the discrepancy seminorm, we have
\begin{equation}
    \begin{split}
        \|\ma\|^{\delta}_{(k)} & = \max_{\mm \in \textup{Conv}(\mathcal{P}^0_{k}(n))} \R \trace(\ma \mm ) \\
        & = \max_{\substack{\mm \in \textup{Conv}(\mathcal{P}_{k}(n))}} \min_{\alpha \in \mathbb{C}} \; \R \trace(\ma \mm ) - \R \alpha \trace(\mm) \\
        & = \min_{\alpha \in \mathbb{C}} \max_{\mm \in \textup{Conv}(\mathcal{P}_{k}(n))} \; \R \trace((\ma-\alpha \mi_n) \mm ) \\
        & = \min_{\alpha \in \mathbb{C}}\|\ma - \alpha \mi_n\|_{(k)},
    \end{split}
\end{equation}
where the second equality follows since $\textup{Conv}(\mathcal{P}^0_{k}(n))$ is the intersection of the traceless matrices with $\textup{Conv}(\mathcal{P}_{k}(n))$. The third equality follows by Sion's minimax theorem, while the last equality follows from the maximal representation~\eqref{cmp_Kynorm}.
\end{proof}

In words, Theorem~\ref{lemma_min} shows that discrepancy values can be thought of as a vector-valued extension of the (spectral norm) distance of a linear operator to a scalar multiple of the identity; i.e., a vector-extension of the following observation
\begin{equation}
   \delta_1(\ma) = \min_{\alpha \in \mathbb{C}}\;\sigma_1(\ma-\alpha \mi_n) = \max_{\substack{\|\vx\|=\|\vy\|=1 \\ \vx \bot \vy}}|\langle \ma\vx,\vy \rangle|.
\end{equation}
The problem of projecting a linear operator onto the subspace of scalar linear operator has been investigated in the literature (see~\cite{Apostol84,bhatia1999orthogonality}). Finally, the two relations between discrepancy seminorms and Ky-Fan norms imply the following identity.

\begin{corollary}
Using Lemma~\ref{max_delta} and Theorem~\ref{lemma_min}, we have the equality
\[
    \max_{\mU \in \mathbf{U}(n)}\min_{\alpha \in \mathbb{C}} \|\ma-\alpha\mU \|_{(k)} = \|\ma\|_{(k)}.
\]
\end{corollary}

\subsection{Discrepancy values for Hermitian matrices}
\label{sec:discHerm}
Now we can illustrate the connection between discrepancy values and the spectral spread of Hermitian matrices. When $\ma$ is Hermitian, we can solve the optimization problem~\eqref{eq:minchar} in closed form and get $\delta_{2k-1}(\ma)=\delta_{2k}(\ma)=|\lambda_k(\ma)-\lambda_{n-k+1}(\ma)|/2$
for $k=1,\dots,\floor*{n/2}$. Moreover, if $n$ is odd, then $\delta_n(\ma)$ would be zero. On the other hand, in Def.~\ref{Def:radii}, for a Hermitian matrix with eigenvectors $\vv_1,\vv_2,\dots, \vv_n$ corresponding to the non-increasing eigenvalues, we can check that the vectors $\vx_{2k-1}=(-\vv_{k}+\vv_{n-k+1})/\sqrt{2}$ and $\vy_{2k-1}=(-\vv_{k}-\vv_{n-k+1})/\sqrt{2}$ for the odd terms, and $\vx_{2k}=\vy_{2k-1}$ and $\vy_{2k}=\vx_{2k-1}$ for the even terms are  maximizers in problem~\ref{Def:radii}. More compactly, for a Hermitian matrix $\ma$ we have the (vector) equality
\begin{equation}
    \vdelta(\ma) = \frac{|\vlambda^{\downarrow}(\ma)-\vlambda^{\uparrow}(\ma)|^{\downarrow}}{2}.
\end{equation}
\begin{remark}
  The discrepancy value of Hermitian matrices is invariant w.r.t.\ a particular transformation. We call this invariance \emph{looseness} or \emph{slackness}. It states that if we fix the outer eigenvalues, we can shift the inner eigenvalues, and the discrepancy values remain the same so long as we do not cross the fixed eigenvalues. We are going to exploit this property later in this paper.
\end{remark}

The interlacing property for Hermitian matrices implies the stability of discrepancy values under perturbation.
\begin{proposition}[Interlacing theorem]
Let $\ma$ be a Hermitian matrix of order $n$, and $\mb$ be a principal submatrix of $\ma$. Using the Cauchy interlacing theorem we obtain $\delta_1(\ma)=\delta_2(\ma)\geq \delta_1(\mb)=\delta_2(\mb)\geq \delta_3(\ma)=\delta_4(\ma)\geq \delta_3(\mb)=\delta_4(\mb)\geq \dots$.
\end{proposition}

Finally, we note the following: for an $n \times n$ Hermitian matrix $\ma$, when $n$ is even, the spectral spread of the direct sum $\ma\oplus \ma$ is related to the discrepancy values via
\begin{equation}
    \label{eq:delta_spr}
    \vdelta(\ma) = \text{spr}^+(\ma \oplus \ma)/2,
\end{equation}
while for odd $n$, $\vdelta(\ma)$ has an extra zero at the end of the vector $\text{spr}^+(\ma \oplus \ma)/2$. 

\subsection{Important classes of matrices based on their discrepancy values}
\label{sec:subclasses}
We know that unitary matrices have unit singular values, while partial isometries have zero and one as their singular values, and both matrices play an important role in the characterization of singular values. We believe that to understand discrepancy values better, we need to understand the equivalent classes of matrices associated with them. We explore below the structure of matrices with unit discrepancy values and discuss matrices with or zero and one discrepancy values. Before that, we need some tools.

\begin{proposition}
\label{singeqdisc}
For $\ma \in \C^{n\times n}$, we have $\vdelta(\ma)=\vsigma(\ma) $ iff it has the singular value decomposition $\ma = \mU \Sigma \mv^*$, with $\mU,\mv \in \mathbf{U}(n)$ and $\Diag(\mU^*\mv)=0$; i.e., $\langle \vu_i,\vv_i \rangle=0$ for $i=1,\dots, n$.
\end{proposition}
\begin{proof}
The ``if'' part follows trivially by the definition of discrepancy values. We first prove the ``only if'' for full-rank matrices with unique singular values; in this case, the singular vectors are unique up to multiplication by $e^{i\theta}$. We then have
\begin{equation*}
    \max_{\substack{\|\vx_{11}\|=\|\vy_{11}\|=1 \\ \vx_{11} \bot \vy_{11}}}|\langle \ma\vx_{11},\vy_{11} \rangle| = \max_{\substack{\|\vu_{11}\|=\|\vv_{11}\|=1}}|\langle \ma\vu_{11},\vv_{11} \rangle|.
\end{equation*}
Therefore, it must be the case that the optimizers $\vu^*_{11}=e^{i\theta_1}\vx^*_{11}$ and $\vv^*_{11}=e^{i\phi_1}\vy^*_{11}$; i.e., $\langle \vu^*_{11},\vv^*_{11} \rangle=0$. Next, we have the equality
$$
    \max_{\substack{\|\vx_{12}\|=\|\vy_{12}\|=1 \\ \|\vx_{22}\|=\|\vy_{22}\|=1 \\ \vx_{12}\bot \vx_{22}\\ \vy_{12} \bot \vy_{22} \\ \langle \vx_{12},\vy_{12} \rangle=-\langle \vx_{22},\vy_{22} \rangle}} \big|\langle \ma\vx_{12},\vy_{12} \rangle + \langle \ma\vx_{22},\vy_{22} \rangle\big| =  \max_{\substack{\|\vu_{12}\|=\|\vv_{12}\|=1 \\ \|\vu_{22}\|=\|\vv_{22}\|=1 \\ \vu_{12} \bot \vu_{22}\\ \vv_{12} \bot \vv_{22}}}  \big|\langle \ma\vu_{12},\vv_{12} \rangle + \langle \ma\vu_{22},\vv_{22} \rangle\big|,
$$
which implies that $\langle \vu^*_{12},\vv^*_{12} \rangle+\langle \vu^*_{22},\vv^*_{22} \rangle=0$. On the other hand, we know that $\vu^*_{12} = e^{i\theta_2} \vu^*_{11}$ and $\vv^*_{12} = e^{i\phi_2} \vv^*_{11}$; thus $\langle \vu^*_{22},\vv^*_{22} \rangle=0$. The general result follows inductively. For matrices with repeated singular values or zero singular values, not all possible singular value decompositions of $\ma$ have the property that $\Diag(\mU^*\mv)=0$, but a decomposition with such property belongs to the set of possible singular value decompositions of that matrix.
\end{proof}
\begin{remark}
In other words, matrices of the form $\ma=\sum_{i=1}^n \delta_i \mx_i+ \alpha \mi_n$ have discrepancy values equal to $\delta_i$, where $\delta_i\geq 0$, and when $\mx_i$ are nilpotent rank-1 isometries and mutally orthogonal; i.e., $\langle \mx_i,\mx_j \rangle=0$.
\end{remark}
\begin{corollary}
\label{corrDelta10}
If $\ma=\mU\mv^*+\alpha\mi_n$, where $\mU$ and $\mv$ are unitary matrices with $\trace(\mU^*\mv)=0$, then $\delta_i(\ma)=1$.
\end{corollary}
\begin{proof}
It follows by the fact that there exists a unitary matrix $\mq$ such that for $\tilde{\mU}=\mU\mq$ and $\tilde{\mv}=\mv\mq$ we have $\Diag(\tilde{\mU}^*\tilde{\mv})=0$ and $\ma=\tilde{\mU}\tilde{\mv}^*+\alpha\mi_n$.
\end{proof}
\begin{corollary}
The singular values and the discrepancy values of matrices of the form $\mb=\mq\ma\mq^*$, where $\ma$ is an antidiagonal matrix and $\mq$ unitary, are equal to the absolute value of the anti-diagonal entries of $\ma$. 
\end{corollary}
We next note a concrete setting where discrepancy values equal singular values.
\begin{proposition}
  \label{prop:hamil}
For a $2n \times 2n$ Hamiltonian matrix $\mh$, we have
\[
\vdelta(\mh)=\vsigma(\mh).
\]
\end{proposition}
\begin{proof}
We know that any Hamiltonian matrix can be represented as the multiplication of $\mj={\scriptsize\begin{bmatrix}
0 & \mi_n \\ -\mi_n & 0
\end{bmatrix}}$ and a symmetric matrix. Thus, we have $\mh = \mj \mq \mLambda\mq^*$. Let $\mLambda = \msigma \mD$ where $\mD$ is just a diagonal matrix with $\pm 1$ entries. 
In the singular value decomposition of $\mh=\mU \msigma \mv^*$ we have $\mu=\mj\mq$ and $\mv=\mq \mD$. To show that $\Diag(\mD\mq^*\mj\mq)=0$ one only needs to prove $\Diag(\mq^*\mj\mq)=0$, which is immediate after recognizing that $\vdelta(\mq^*\mj\mq)=\vsigma(\mq^*\mj\mq)$.
\end{proof}
\begin{remark}
  The same argument can be used to prove that the singular and discrepancy values of the multiplication of $\mj$ and any normal matrix are equal.
\end{remark}

Let us adopt the notation $\Psi(n)$ to refer to the set of $n\times n$ matrices with unit discrepancy values. Theorem~\ref{lemma_all1Delts} states that this class is equivalent to the scalar shifts of traceless unitary matrices.

\begin{theorem}
\label{lemma_all1Delts}
A square matrix belongs to $\Psi(n)$ iff it has the form $\mm-\alpha \mi_n$, where $\mm$ is unitary and $\trace\mm=0$, while $\alpha$ an arbitrary complex number.
\end{theorem}
\begin{proof}
We already proved in Corollary~\ref{corrDelta10} that $\mm-\alpha \mi_n \in \Psi(n)$ when $\mm$ is a traceless unitary matrix. 
By the definition of discrepancy values, there exists $\hat{\alpha}$ such that $\sigma_1(\ma-\hat{\alpha}\mi_n) = \delta_1(\ma)=1$, and $\sigma_1(\ma-\hat{\alpha}\mi_n) + \sigma_2(\ma-\hat{\alpha}\mi_n) \geq \delta_1(\ma)+\delta_2(\ma)=2$. Therefore, $\sigma_2(\ma-\hat{\alpha}\mi_n) \geq 1$, but we assumed that $\sigma_2(\ma-\hat{\alpha}\mi_n) \leq \sigma_1(\ma-\hat{\alpha}\mi_n)$; hence $\sigma_2(\ma-\hat{\alpha}\mi_n)=\delta_2(\ma)=1$. Using the same argument we can show that $\sigma_k(\ma-\hat{\alpha}\mi_n)=\delta_k(\ma)=1$, for $1\leq k \leq n$. Finally, by proposition~\ref{singeqdisc} we know that matrix $\ma -\hat{\alpha}\mi_n$ has a decomposition $\mU\mv^*$ where $\Diag(\mU^*\mv)=0$.
\end{proof}
In other words, a matrix is in $\Psi(n)$ iff for some $\alpha$ its eigenvalues satisfy $|\lambda_i-\alpha|=1$ and $\sum_{i=1}^n\lambda_i=0$, which always has roots for $n\geq 2$. Thus, $\Psi(n)$ is not empty.

Before studying the class of matrices with zero or one discrepancy values, let us define the notion of \emph{discrepancy-rank}.
\begin{definition} Let $\ma \in \C^{n\times n}$. We define \emph{discrepancy-rank} of $\ma$ as
\begin{equation}
    r^{\delta}(\ma) := \text{the number of non-zero entries of } \vdelta(\ma).
  \end{equation}
\end{definition}
One can verify that $r^{\delta}(\ma) = r(\ma-\bar{\alpha}\mi_n)$ where $\bar{\alpha} = \argmin_{\alpha \in \mathbb{C}}\|\ma-\alpha\mi_n\|_{(n)}$.
We can also see that $r^{\delta}(\ma)=0$ iff $\ma=\alpha \mi_n$ for a complex number $\alpha$. And $r^{\delta}(\ma)=1$ iff $\ma = \beta \vu \vv^*+\alpha \mi_n$ for some (complex) $\alpha$ and $\beta$, and orthogonal unit vectors $\vu$ and $\vv$.

\begin{definition}
Let $\Psi_{k}(n)$ denote the set of $n \times n$ matrices with discrepancy-rank $k$ with unit non-zero discrepancy values (i.e. $\delta_i=1$ for $i=1,\dots, k$).
\end{definition}

Now we want to explore the relation between $\Psi_{k}(n)$ and the set $\mathcal{P}^0_{k}(n)$. Note that $\mathcal{P}^0_{k}(n) \neq \Psi_{k}(n)$; for instance, consider $\mm=\diag([1,3,4])$\footnote{ or any scalar shift of it.} which belongs to $\Psi_{2}(3)$ but not to $\mathcal{P}^0_{2}(3)$. However, we can state the following inclusion:
\begin{lemma}
    $\mathcal{P}^0_{k}(n) \subset \Psi_{k}(n)$.
\end{lemma}
\begin{proof}
Consider a matrix $\mm \in \mathcal{P}^0_{k}(n)$. Observe that $\|\mm\|_{k}^{\delta} = k$ by using the left and right singular vectors of $\mm$ as the vectors $\vx$ and $\vy$ in~\eqref{main_def}. On the other hand, we know that $\vdelta(\mm) \prec_w \mathbf{1}_k$. Therefore, the only possibility is that $\vdelta(\mm)=\mathbf{1}_k$.
\end{proof}

\subsection{Basic majorization inequalities}
\label{sec:basicMajor}
A plethora of majorization inequalities for singular values exists in the literature. In this part, we investigate some majorization results for discrepancy values. 
We also see that the discrepancy values satisfy many of the majorization bounds known for the spectral spread of Hermitian matrices, which underscores our claim that discrepancy values are a true generalization of spectral spread. First, let us state two trivial inequalities.

\begin{proposition}
For $n \times n$ matrices $\ma$ and $\mb$, one has
\begin{enumerate}
\label{prop:easy2prove}
    \item $|\vdelta(\ma)-\vdelta(\mb)| \prec_w \vdelta(\ma\pm\mb) \prec_w \vdelta(\ma)+\vdelta(\mb)$.
    \item $\vdelta(\ma) \prec_w \vsigma(\ma)$.
\end{enumerate}
\end{proposition}
Now let us look at some non-trivial inequalities. The first one is about the pinching of matrices (see e.g., Problem II.5.5 in ~\cite{bhatia2013matrix} for the definition). A similar inequality holds for singular values, see~\cite{bhatia2008commutators}.
\begin{proposition}
  \label{prop:pinching}
Let $\ma$ be any square matrix and $\mathcal{C}(\ma)$ denote a pinching, and $\mathcal{L}(\ma)=A-\mathcal{C}(\ma)$ an anti-pinching of this matrix. Then, we have
\begin{align*}
  \vdelta\big(\mathcal{C}(\ma)\big) &\prec_{w} \vdelta(\ma)\\
  \vdelta\big(\mathcal{L}(\ma)\big) &\prec_{w} \vdelta(\ma).
\end{align*}
\end{proposition}
\begin{proof}
We only prove the first inequality; the second one is similar. Since $\mathcal{C}(\ma)$ is a pinching, there exists a unitary $\mU$ (see~\cite{bhatia2008commutators}) such that $\mathcal{C}(\ma)=\frac{1}{2}(\ma+\mU\ma\mU^*)$. Thus,
\begin{equation}
    \begin{split}
    \sum_{i=1}^k \delta_i\big(\mathcal{C}(\ma)\big) & = \frac{1}{2}\sum_{i=1}^k\delta_i\big(\ma+\mU\ma\mU^*\big) \leq \sum_{i=1}^k\delta_i(\ma).
    \end{split}
\end{equation}
\end{proof}
% \begin{proposition}
% \label{prop:antipinching}
% Let $\ma$ be a square matrix and $\mathcal{L}(\ma)$ denote the anti-pinching of this matrix; namely $\mathcal{L}(\ma)=A-\mathcal{C}(\ma)$, then we have
% \[
%     \vdelta\big(\mathcal{L}(\ma)\big) \prec_{w} \vdelta(\ma)
% \]
% \end{proposition}
% \begin{proof}
% Noting that $\mathcal{C}(\ma)=\frac{1}{2}(\ma+\mU\ma\mU^*)$ for some unitary matrix $\mU$, we have
% \begin{equation}
%     \begin{split}
%     \sum_{i=1}^k \delta_i\big(\mathcal{L}(\ma)\big) & = \frac{1}{2}\sum_{i=1}^k\delta_i\big(\ma-\mU\ma\mU^*\big) \leq \sum_{i=1}^k\delta_i(\ma).
%     \end{split}
% \end{equation}
% \end{proof}

\begin{corollary}
\label{corr:block}
For any $n \times n$ matrices $\ma_1, \ma_2, \ma_3, \ma_4$, we have
\[
    \vsigma \left(\begin{bmatrix}
    0 & \ma_2 \\
    \ma_1 & 0
    \end{bmatrix}\right) = \vdelta\left(\begin{bmatrix}
    0 & \ma_2 \\
    \ma_1 & 0
    \end{bmatrix}\right) \prec_w \vdelta\left(\begin{bmatrix}
    \ma_3 & \ma_2 \\
    \ma_1 & \ma_4
    \end{bmatrix}\right).
\]
\end{corollary}
\begin{proof}
Assume that $\ma_1$ and $\ma_2$ have the singular value decompositions $\ma_1=\mU_1\msigma_1\mv_1^*$ and $\ma_2=\mU_2\msigma_2\mv_2^*$, then we have
\[
    \begin{bmatrix}
    0 & \ma_2 \\
    \ma_1 & 0
    \end{bmatrix} = \underbrace{\begin{bmatrix}
    0 & \mU_2 \\
    \mU_1 & 0
    \end{bmatrix}}_{\mq}
    \begin{bmatrix}
    \msigma_1 & 0 \\
    0 & \msigma_2 
    \end{bmatrix}
    {\underbrace{\begin{bmatrix}
    \mv_1 & 0 \\
    0 & \mv_2
    \end{bmatrix}}_{\mr}}^*
\]
Note that $\mq$ and $\mr$ are unitary matrices and $\Diag(\mq^*\mr)=0$; therefore by Proposition~\ref{singeqdisc} we have the equality. The weak majorization follows by Proposition~\ref{prop:pinching}.
\end{proof}

\begin{remark}
  One can recover the non-commutative AM-GM inequality~\cite{bhatia1990singular} from Corollary~\ref{corr:block}. This corollary implies that for any square matrices $\ma$ and $\mb$,
  \begin{equation*}
    \vsigma \left(\begin{bmatrix}
    0 & \ma\mb^* \\
    \mb\ma^* & 0
    \end{bmatrix}\right) \prec_w \vdelta\left(\begin{bmatrix}
    \ma & 0 \\
    \mb & 0
    \end{bmatrix} \begin{bmatrix}
    \ma^* & \mb^* \\
    0 & 0
    \end{bmatrix}\right).
  \end{equation*}
  Using the fact that for $2n \times 2n$ positive definite matrix $\mx$ we have $2\delta_{2i-1}(\mx) \leq \sigma_i(\mx)$, for $i=1,\dots, n$, we get the inequality
  \begin{equation*}
    2\sigma_{2i-1} \left(\begin{bmatrix}
    0 & \ma\mb^* \\
    \mb\ma^* & 0
    \end{bmatrix}\right) \leq \sigma_i\left(\begin{bmatrix}
    \ma^*\ma+\mb^*\mb & 0 \\
    0 & 0
    \end{bmatrix}\right),
  \end{equation*}
  which implies that for $i=1,\ldots,n$,
  \[
  2\sigma_i(\ma\mb^*) \leq \sigma_i(\ma^*\ma+\mb^*\mb).
  \]
\end{remark}

Corollary~\ref{corr:block} improves and generalizes Inequality (8) in~\cite{massey2021norm} that studies spectral spread. Moreover, it also implies that
\begin{corollary}
\label{corl:perp}
For an isometry $\ms \in \C^{n\times k}$ and arbitrary matrix $\mx$ we have
\[
    \vsigma(\ms^*\mx\ms_{\perp}) \prec_w \restr{\vdelta(\mx)}{\min{(k,n-k)}}\,,
\]
where $\ms_{\perp}$ is an isometry whose range is orthogonal to the range of $\ms$.
\end{corollary}  
Using Corollary~\ref{corl:perp} we can provide an easy proof to a known upper-bound on the principal angles $\Theta (\mathcal{S}, \mathcal{T})$ between two subspaces $\mathcal{S}$ and $\mathcal{T}$. For a detailed discussion regarding the principal angles between subspaces, we refer the readers to~\cite{massey2021spectral}, in which this upper bound was first discovered.
 
 \begin{theorem} [Theorem 2.15 in~\cite{massey2021spectral}]
 \label{Thm:principalAngle}
 Given the Hermitian matrix $\mx$, the principal angles between k dimensional subspaces $\mathcal{S} \subset \mathbb{C}^n$ and $\mathcal{T}=e^{i\mx} \mathcal{S} \subset \mathbb{C}^n$ satisfy
 \[
    \Theta (\mathcal{S}, \mathcal{T}) \prec_w \restr{\vdelta(\mx)}{\min{(k,n-k)}},
 \]
 where $\restr{\vdelta(\mx)}{k}$ denotes a $k$ dimensional vector with the first $k$ entries of $\vdelta(\mx)$.
 \end{theorem}
 \begin{proof}
 Define $\mathcal{T}(t)=e^{it\mx}\mathcal{S}$. We have $\mathcal{T}(0) = \mathcal{S}$ and $\mathcal{T}(1) = \mathcal{T}$. By the triangle inequality for principal angles we have
 \begin{equation}
     \begin{split}
          \Theta(\mathcal{S}, \mathcal{T}) & \prec_w \sum_{j=0}^{m-1} \Theta\big(\mathcal{T}(\tfrac{j}{m}),\mathcal{T}(\tfrac{j+1}{m})\big) = \sum_{j=0}^{m-1} \arcsin\Big(\vsigma(\ms^* e^{\tfrac{i\mx}{m}}\ms_{\perp})\Big) \\
          & = m \arcsin\Big(\vsigma(\ms^* e^{\tfrac{i\mx}{m}}\ms_{\perp})\Big) \\
          & \myto \vsigma(\ms^* \mx\ms_{\perp})\\ &\prec_w \restr{\vdelta(\mx)}{\min{(k,n-k)}},
     \end{split}
 \end{equation}
where the last majorization follows by Corollary~\ref{corl:perp}. The limit can be justified by L'Hopital's rule and the continuity of singular values.
 \end{proof}
 
Finally, we propose two inequalities that shed further light on the relationship between  discrepancy values and singular values.
\begin{lemma}
For the $n \times n$ Hermitian matrix $\mx$ we have the majorization
\[
    \vdelta(e^{i\mx})\prec_w \vsigma(\mx).
\]
\end{lemma}
\begin{proof}
We first show that
%\begin{equation}
%\label{ineq:proofHerm}
    $\vdelta(e^{i\mx})\prec_w \int_{0}^{1} \vdelta(i\mx e^{i t\mx}) dt$. 
%\end{equation}
Then, using the fact that $\vdelta(ie^{i t\mx}\mx)\prec_w \vsigma(\mx)$ we have
$\vdelta(e^{i\mx})\prec_w \int_{0}^{1} \vsigma(\mx) dt = \vsigma(\mx)$. To prove the former inequality, first define $f(t) = e^{it\mx}$. Then, consider the series
\[
    e^{i\mx} = \mi_n - \nlsum_{j=1}^{m-1} f\left(\tfrac{j}{m}\right) - f\left(\tfrac{j+1}{m}\right).
\]
Consequently, upon applying $\vdelta(\cdot)$ to both sides, we obtain
\[
    \vdelta(e^{i\mx}) \prec_w \nlsum_{j=1}^{m-1} \vdelta\left(f\left(\tfrac{j}{m}\right) - f\left(\tfrac{j+1}{m}\right)\right).
\]
If we let $m$ goes to infinity the RHS would converge to $\int_{0}^{1} \vdelta\left(\tfrac{d}{dt} f(t)\right) dt$. Finally, use the fact that $df(t)/dt = i \mx e^{it\mx}$.
\end{proof}

\begin{corollary}
For any square matrix $\ma$, let $|\ma|=(\ma^*\ma)^{1/2}$. Then, we have
\[
    \vdelta(e^{i|\ma|})\prec_w \vsigma(\ma).
\]
%where $|\ma| = $.
\end{corollary}
We also have another similar relationship between the discrepancy of the absolute values of a matrix and the discrepancy of the matrix.
For any matrix $\ma$,
\begin{equation}\label{eq:6}
  \vdelta(|\ma|) = \frac{|\vsigma^{\downarrow}(\ma)-\vsigma^{\uparrow}(\ma)|^{\downarrow}}{2} \prec_w \vdelta(\ma).
\end{equation}
To prove inequality~\eqref{eq:6}, we need the following statement:
\begin{proposition}
Let $\mP$ be Hermitian positive definite, $\mq$ be unitary. Then, 
\[
    \vdelta(\mP) \prec_w \vdelta(\mq \mP).
\]
\end{proposition}
\begin{proof}
It suffices to show that for any diagonal matrix with non-negative entries $\mLambda$ and unitary matrix $\mq$
we have 
\[
    \vdelta(\mLambda) \prec_w \vdelta(\mq \mLambda).
\]
Note that $\vsigma(\mLambda)-|\alpha| \vsigma(\mq) \prec_w \vsigma(\mLambda - \alpha \mq)$; hence for any unitary matrix $\mq$, we have $\min_{\alpha \in \mathbb{R}^+} \vertiii{\mLambda - \alpha \mi} \leq \min_{\alpha \in \mathbb{C}} \vertiii{\mLambda - \alpha \mq}$, which implies the desired inequality.
\end{proof}

\subsection{Discrepancy values of direct sums}
\label{sec:directsum}
Below, we briefly discuss discrepancy values of the direct sum of matrices. These results would particularly become useful when comparing the results for discrepancy values with ones for the spread of Hermitian matrices.
Inspired by the minimal representation of the discrepancy seminorm, we introduce for $\ma,\mb \in \C^{n\times n}$, the generalization
\begin{equation}
    \sum_{i=1}^k \delta_{i}(\ma,\mb) := \min_{\alpha\in\mathbb{C}}\;\frac{\|\ma-\alpha\mi_n\|_{(k)}+\|\mb-\alpha\mi_n\|_{(k)}}{2}.
\end{equation}
Note that $\vdelta(\ma, \mq^*\ma\mq)=\vdelta(\ma)$, for any unitary matrix $\mq$, and $2\vdelta(\ma,0)=\vsigma(\ma)$.

\begin{proposition}
  \label{prop:deltaab}
  Let $\ma, \mb \in \C^{n\times n}$. The following majorization holds:
  \begin{equation}
    \label{eq:8}
    (\vdelta(\ma,\mb), \vdelta(\ma,\mb)) \prec_w \vdelta(\ma\oplus \mb).
  \end{equation}
\end{proposition}
\begin{proof}
First note that for any $\vx,\vy \in \mathbb{R}^n$, we have
\[
    \Big(\frac{\vx^{\downarrow}+\vy^{\downarrow}}{2}, \frac{\vx^{\downarrow}+\vy^{\downarrow}}{2}\Big) \prec_w (\vx, \vy)^{\downarrow}.
\]
Then for any scalar $\alpha$, we have
\[
    \Big(\frac{\vsigma(\ma -\alpha\mi_n)^{\downarrow}+\vsigma(\mb -\alpha\mi_n)^{\downarrow}}{2}, \frac{\vsigma(\ma -\alpha\mi_n)^{\downarrow}+\vsigma(\mb -\alpha\mi_n)^{\downarrow}}{2}\Big) \prec_w \vsigma(\ma\oplus\mb -\alpha\mi_n)^{\downarrow}.
\]
The result follows by minimizing the RHS of the above inequality w.r.t. $\alpha$.
\end{proof}

For the sequel, we need to introduce an averaging operator on vectors. Assume that $\vx \in \mathbb{R}^{kn}$, where $k$ and $n$ are some positive integers. Then, define $\mu_{k}(\vx) \in \mathbb{R}^n$ by
\begin{equation}
    \mu_{k}(\vx) := \begin{bmatrix} \frac{\vx_1+\dots+\vx_k}{k} \\ \frac{\vx_{k+1}+\dots+\vx_{2k}}{k} \\ \vdots \\ \frac{\vx_{n-k+1}+\dots+\vx_{n}}{k} \end{bmatrix}.
\end{equation}
Using the operator $\mu_k$ we can restate \eqref{eq:8} as $\vdelta(\ma,\mb) \prec_w \mu_2(\vdelta(\ma\oplus \mb))$, or restate the connection between the spectral spread of a Hermitian matrix $\ma \in \mm_n$ with even $n$ and discrepancy values as $\mu_2(\vdelta(\ma))=\text{spr}^+(\ma)/2$. Therefore, for Hermitian matrices $\ma$ and $\mb$, we have $\vdelta(\ma,\mb) \prec_w \text{spr}^+(\ma\oplus \mb)/2$. 
%The following inequality is also evident:
% \begin{proposition}
% Let $\ma$ be Hermitian positive definite. Then, 
% \[
%     \mu_2(\vdelta(\ma)) \leq \mu_2(\vsigma(\ma)).
% \]
% \end{proposition}
The averaging operator allows us to write discrepancy of direct sums of a matrix with itself via the discrepancy of the individual matrix. Indeed, we have the following:
\begin{lemma}
Let $\ma$ be $n \times n$ matrix. Then
\[
    \mu_k \big(\vdelta(\,\overbrace{\ma \oplus \cdots \oplus \ma}^{k}\,)\big) = \vdelta(\ma).
\]
\end{lemma}
\begin{proof}
% Since we have
% \[
%     \|\,\overbrace{\ma \oplus \cdots \oplus \ma}^{k}\,\|_{(km)} = k \|\ma\|_{(m)}, \qquad m=1, \dots, n,
% \]
Note that
\begin{equation*}
    \begin{split}
        \min_{\alpha\in\mathbb{C}} \|\,\overbrace{\ma \oplus \cdots \oplus \ma}^{k}\, -\alpha\mi_n\|_{(km)} &= \min_{\alpha\in\mathbb{C}} \|\,\overbrace{(\ma-\alpha\mi_n) \oplus \cdots \oplus (\ma-\alpha\mi_n)}^{k}\,\|_{(km)} \\
        & = k\, \min_{\alpha\in\mathbb{C}} \|\ma -\alpha\mi_n\|_{(m)},
    \end{split}
\end{equation*}
for $m=1, \dots, n$. Hence we have the desired result.
\end{proof}
We can similarly prove this lemma if matrices $\ma$ are arbitrarily replaced by $\ma^*$.

%%% Local Variables:
%%% mode: latex
%%% TeX-master: "main"
%%% End:

\section{Application: bounding norms of commutators}\label{sec: application}
In this section, our main goal is to employ the tools developed so far to bound the Ky-Fan norms of  commutators. We first note an obvious invariance of the \emph{generalized commutator} $\ma\mx-\mx\mb$ under shifts: for any $\alpha\in \C$ and square matrices $\ma,\mb,\mx$, we have
\begin{equation}
    (\ma-\alpha\mi_n)\mx-\mx(\mb-\alpha\mi_n) = \ma\mx-\mx\mb.
    \label{eq:inv}
\end{equation}
An special case of this invariance is $[\ma,\mb] = [\ma, \mb - \alpha \mi_n]$, where $[\cdot , \cdot]$ denotes the usual commutator (Lie bracket). We will use invariance~\eqref{eq:inv} together with properties of discrepancy values to amplify existing majorization inequalities such as $\vsigma([\ma,\mb]) \prec_w 2\vsigma(\ma)\vsigma(\mb)$, and to achieve sharp bounds for similar objects. Our first result is:

\begin{lemma}
\label{lemma:pi2}
For square matrices $\ma, \mb$, and any partial isometry $\mx$ we have
\[
\vsigma (\ma\mx-\mx\mb ) \prec_w 2 \vdelta(\ma, \mb)\vsigma(\mx).
\]
\end{lemma}
\begin{proof}
Using invariance~\eqref{eq:inv}, we know that $\sum_{i=1}^k \sigma_i (\ma\mx-\mx\mb) \leq \sum_{i=1}^k (\sigma_i(\ma - \alpha \mi) + \sigma_i(\mb - \alpha \mi))\sigma_i(\mx)$ for any $\alpha\in \mathbb{C}$. Now we minimize the RHS for any $k=1,\dots, n$. Considering the fact that $\sigma_i(\mx)$ is either $1$ or $0$, we get the desired result.
\end{proof}
Next, we propose a decomposition that enables us to extend the above lemma.
\begin{lemma}
  \label{lem:decomp}
Any square matrix $\ma$ has the following decomposition
\begin{equation}
\label{eq:10}
\ma=\sum_{i=1}^n \alpha_i \mx_i,
\end{equation}
where $\mx_i$ is a rank-$i$ isometry; i.e., with $i$ non-zero singular values equal to one. And, $\sum_{i=1}^n \alpha_i=\sigma_1$, $\sum_{i=2}^n \alpha_i=\sigma_2, \dots, \alpha_n=\sigma_n$. %, where the singular values are ordered in a decreasing manner.
\end{lemma}
\begin{proof}
Consider the singular value decomposition $\ma=\sum_{i=1}^n \sigma_i \vx_i\vy_i^*$, and then use  summation by parts. Thus, we can rewrite $\ma$ as follows:
\[
\ma=(\sigma_1-\sigma_2)\vx_1\vy_1^*+(\sigma_2-\sigma_3)(\vx_1\vy_1^*+\vx_2\vy_2^*)+\dots+\sigma_n(\vx_1\vy_1^*+\dots+\vx_n\vy_n^*).
\]
Now let $\alpha_i=\sigma_i-\sigma_{i+1}$ and $\mx_i= \vx_1\vy_1^*+\dots+\vx_i\vy_i^*$, completing the proof.
\end{proof}
Observe that if matrix $\ma$ in Lemma~\ref{lem:decomp} is positive definite, then each $\mx_i$ is an orthogonal projection. We are ready to now present the first main result of this section.
\begin{theorem}
  \label{thm:commuteMain}
  For $n \times n$ square matrices $\ma$, $\mb$, and $\mx$ we have
\begin{equation}
\label{ineq:main}
    \vsigma (\ma\mx-\mx\mb) \prec_w 2 \vdelta(\ma,\mb)\vsigma(\mx).
\end{equation}
\end{theorem}
\begin{proof}
Using Lemma~\ref{lem:decomp} to obtain the decomposition $\mx=\sum_{i=1}^n \alpha_i \mx_i$, we have
\begin{equation}
    \begin{split}
        \vsigma(\ma\mx-\mx\mb)&=\vsigma\Big( \nlsum_{i=1}^n \alpha_i ( \ma\mx_i-\mx_i\mb) \Big) \\
        &\prec_w \nlsum_{i=1}^n \alpha_i \vsigma ( \ma\mx_i-\mx_i\mb) \\
        &\prec_w 2\vdelta(\ma, \mb)\sum_{i=1}^n \alpha_i \vsigma(\mx_i) \\
        &= 2\vdelta(\ma, \mb)\vsigma(\mx),
    \end{split}
\end{equation}
where the first majorization follows upon noting that $\alpha_i\geq 0$, and the second majorization from Lemma~\ref{lemma:pi2}.
\end{proof}

Theorem~\ref{thm:commuteMain} provides an extension to a related inequality for the spectral spread proved in~\cite[inequality (11)]{massey2021norm}; using $\vdelta(\ma, \ma)=\vdelta(\ma)$, we have the following:
\begin{corollary}
\label{coroll:zero1}
For the $n \times n$ square matrices $\ma$ and $\mb$ we have
\begin{equation}
  \label{eq:9}
    \vsigma \big([\mb,\ma]\big) \prec_w 2 \vdelta(\mb)\vsigma(\ma).
\end{equation}
\end{corollary}
\begin{remark}
\label{numb}
If either $\ma$ or $\mb$ belongs to $\Psi_k(n)$ (i.e., it only has discrepancy values zero or one) then we can tighten~\eqref{eq:9} to obtain $\vsigma ([\mb,\ma]) \prec_w 2 \vdelta(\mb)\vdelta(\ma)$. % by minimizing the RHS of the above inequalities.
\end{remark}
%Let us continue with inequalities for the generalized commutators.
\begin{corollary}
For $n \times n$ square matrices $\ma$ and $\mb$, we have the inequalities
\begin{align*}
    \vsigma (\ma\mx-\mx\mb) &\prec_w 2 \mu_2\big(\vdelta(\ma \oplus \mb)\big)\vsigma(\mx),\\
    \vsigma (\ma-\mb)       &\prec_w 2 \vdelta(\ma,\mb) \prec_w 2\, \mu_2\big(\vdelta(\ma \oplus \mb)\big).
\end{align*}
\end{corollary}
If $\ma$ and $\mb$ are Hermitian then $2\, \mu_2(\vdelta(\ma \oplus \mb)) = \text{Spr}^+(\ma \oplus \mb)$, and the resulting inequality has been previously shown for spectral spread~\cite[inequality (9)]{massey2021norm}. 
If $\ma$ and $\mb$ are positive definite, then considering the fact that $\delta_{2i-1}(\ma \oplus \mb) \leq \sigma_i(\ma \oplus \mb)$, the above corollary implies that $\sigma_i (\ma\mx-\mx\mb) \leq \sigma_i(\ma \oplus \mb)\sigma_i(\mx)$, for $i=1, \dots, n$.
Moreover, for Hermitian matrices, decomposition~\eqref{eq:10} can be refined as follows.

\begin{lemma}
  \label{lem:herdcomp}
  Any Hermitian matrix $\ma$ has the decomposition
  \begin{equation}
    \label{eq:11}
    \ma=\omega \my+\sum_{i=1}^n \beta_i\mx_i =\omega \my+ (\delta_1-\delta_2)\mx_1+(\delta_2-\delta_3)\mx_2+\cdots +\delta_n\mx_n,
  \end{equation}
  where $\sum_{i=k}^n \beta_i = \delta_{k}(\ma)$, and $\vdelta(\mx_i)$ is a zero-one vector with $k$ ones, $\omega$ is a complex scalar, and $\vdelta(\my)=0$.
\end{lemma}
Note that the last two conditions imply that $\sum_{i=1}^n\vdelta(\mx_i) \beta_i = \vdelta(\ma)$. We now employ the slackness of discrepancy values of Hermitian matrices to prove the claim.
\begin{proof}
Note that for the Hermitian matrix $\ma$ we have $\beta_{2k-1}=0$ since $\delta_{2k-1}=\delta_{2k}$ for $k=1,\dots,\floor{\tfrac{n}{2}}$. First, we can verify that for $\Lambda=\diag(\vlambda(\ma))$ we have
\begin{equation*}
    \begin{split}
        \Lambda =  (\delta_2-\delta_3)&\overbrace{\diag\Big(\big[1,\frac{\alpha_3-\alpha_2}{\delta_2-\delta_3},\cdots,\frac{\alpha_3-\alpha_2}{\delta_2-\delta_3},-1\big]\Big) }^{\mx_2} \\
        &+ (\delta_4-\delta_5)\overbrace{\diag\Big(\big[1,1,\frac{\alpha_5-\alpha_4}{\delta_4-\delta_5},\cdots,\frac{\alpha_5-\alpha_4}{\delta_4-\delta_5},-1,-1\big]\Big)}^{\mx_4} \\ &\qquad\qquad\qquad +\cdots+\delta_{n-1}
\underbrace{\diag\Big(\big[1,\cdots,1,-1,\cdots,-1\big]\Big)}_{\mx_{n-1}}+\alpha_1\underbrace{(\mi+\mo)}_{\my}\; ,
    \end{split}
\end{equation*}
where we assume that $\alpha_1=\alpha_2=\frac{\lambda_1+\lambda_n}{2}$, $\alpha_3=\alpha_4=\frac{\lambda_2+\lambda_{n-1}}{2}, \cdots$, and $\delta_1=\delta_2=\frac{\lambda_1-\lambda_n}{2}$, $\delta_3=\delta_4=\frac{\lambda_2-\lambda_{n-1}}{2}, \cdots$. Matrix $\mo$ is zero when $n$ is even, otherwise only its central entry is nonzero, and we have $\vdelta(\my)=0$. To show that $\vdelta(\mx_2) = \mathbf{1}_2$, $\vdelta(\mx_4) = \mathbf{1}_4, \cdots$, we need to prove that 
\[
|\alpha_i-\alpha_{i-1}| \leq |\delta_i-\delta_{i-1}|
\]
for all odd $i$ greater than or equal to 3. For brevity, we only prove this claim for one case; the other cases follow similarly. We need to show that $|\alpha_3-\alpha_{2}| \leq |\delta_3-\delta_{2}|$, which is equivalent to
\[
|(\lambda_1-\lambda_2)-(\lambda_{n-1}-\lambda_{n})| \leq |(\lambda_1-\lambda_2)+(\lambda_{n-1}-\lambda_{n})|.
\]
This claim follows immediately after recognizing that $\lambda_1-\lambda_2 \geq 0$ and $\lambda_{n-1}-\lambda_n \geq 0$. Finally, the claim of the theorem follows because
\[
\ma = \mq\Lambda\mq^* = (\delta_2-\delta_3)\mq\mx_2\mq^*+(\delta_4-\delta_5)\mq\mx_4\mq^*+\cdots+\delta_{n-1}\mq\mx_{n-1}\mq^*+\alpha_1\mq\my\mq^*,
\]
and because discrepancy values are invariant under unitary similarity transforms.
\end{proof}

With these tools, we are ready to prove the second main result of this section. In particular, Theorem~\ref{thm:herm} fixes one matrix to be Hermitian and proves the majorization inequality~\eqref{wm:generalH} that is tighter than its counterpart implied by Theorem~\ref{thm:commuteMain}.
\begin{theorem}
  \label{thm:herm}
For the $n \times n$ square matrix $\mb$ and Hermitian matrix $\ma$ we have
\begin{equation}
\label{wm:generalH}
    \vsigma ([\ma,\mb]) \prec_w 2 \vdelta(\ma)\vdelta(\mb).
\end{equation}
\end{theorem}
\begin{proof}
Decomposing $\ma$ as per Lemma~\ref{lem:herdcomp}, we have
\begin{equation*}
    \begin{split}
        \vsigma([\mb,\ma])&=\vsigma\big(\big[\mb,\omega \my +\nlsum_{i=1}^n \beta_i \mx_i\big]\big) \\
        &\prec_w |\omega| \vsigma([\mb,\my])+\nlsum_{i=1}^n \beta_i \vsigma([\mb,\mx_i]) \\
        &\prec_w 2\vdelta(\mb)\nlsum_{i=1}^n \beta_i \vdelta(\mx_i) \\
        &= 2\vdelta(\mb)\vdelta(\ma),
    \end{split}
\end{equation*}
where the first majorization follows by the fact that $\beta_i\geq 0$, and the second majorization from Corollary~\ref{coroll:zero1}, and the fact that $\vdelta(\my)=0$.
\end{proof}
In fact we can amplify this theorem to a slightly stronger statement:
\begin{corollary}
\label{generalMajIneq}
For an $n \times n$ matrix $\ma$ and a normal matrix $\mb$ whose eigenvalues lie on a straight line in the complex plane, we have $\vsigma ([\mb,\ma]) \prec_w 2 \vdelta(\mb)\vdelta(\ma)$.
\end{corollary}

Let us now focus on one special case of inequality~\eqref{wm:generalH} where one of the matrices is an orthogonal projection; i.e., $\mP=\mP^*=\mP^2$. This case uncovers a useful inequality between singular and discrepancy values that is known for the spectral spread.
\begin{corollary}
Let $\ma$ be an arbitrary $n$ by $n$ matrix, and $\mP$ be an orthogonal projection of rank $r$. Then for $1\le k\le n$ and $q=\min(2r, 2n-2r)$, we have
\[
\|[\ma, \mP]\|_{(k)} \leq \sum_{i=1}^{\min(q,k)} \delta_i(\ma).
\]
%where $.
\end{corollary}
\begin{corollary}
Given $\mx \in \C^{n\times n}$ and an orthogonal projection $\mP$, we have
\begin{equation}
\label{eq:proj}
    \vsigma(\mP \mx (\mi_n-\mP)) \prec_w \vdelta(\mx).
\end{equation}
\end{corollary}
\begin{proof}
The claim follows from $\vsigma(\mP \mx (\mi_n-\mP)) \prec_w \vsigma([\mP,\mx])$, which holds since the singular values of $[\mP,\mx]$ include all the nonzero singular values of $\mP \mx (\mi_n-\mP)$.
\end{proof}
Inequality~\eqref{eq:proj} was previously known for the spectral spread of self-adjoint operators. Here we show that it holds for arbitrary square matrices. Following the proof of Theorem 4.14 in~\cite{massey2021norm}, the majorization~\eqref{eq:proj} implies the following inequality:
\begin{equation}
    \vsigma(\ma\mx\mb^*) \prec_w \vdelta\big((\ma^*\ma+\mb^*\mb)^{1/2}\mx(\ma^*\ma+\mb^*\mb)^{1/2}\big),
\end{equation}
where $\ma,\mb$ and $\mx$ are arbitrary square matrices.

\subsection{Maximally non-commutative Hermitian matrices}
In this section, we use the majorization bounds derived in the previous part, together with Fan's dominance theorem to solve some optimization problems. These results indicate that the inequalities are sharp. We begin by recalling Fan's celebrated dominance theorem.
\begin{theorem}[Fan's dominance theorem]
Let $\ma,\mb \in \C^{m\times n}$. Then, $\vertiii{\ma} \leq \vertiii{\mb}$ for any unitarily invariant norm on $\C^{m\times n}$ if and only if $\vsigma(\ma) \prec_w \vsigma(\mb)$.
\end{theorem}
The diameter of the unitary orbit of a matrix $\ma \in \C^{n\times n}$ w.r.t.\ the $k\textsuperscript{th}$ Ky-Fan norm is defined by the following optimization problem:
\begin{equation}
    d_k(\ma) := \max_{\mU,\mv \in \mathbf{U}(n)}\|\mv\ma\mv^* - \mU\ma\mU^*\|_{(k)}.
\end{equation}
By inequality~\eqref{ineq:main}, we know that $d_k(\ma) \leq 2 \|\ma\|^{\delta}_{(k)}$. Importantly, for Hermitian matrices we can show that this bound is tight.

\begin{proposition}
[Generalization of Theorem~1.2 in~\cite{bhatia1999orthogonality}] Let $\ma$ be a Hermitian matrix, then $d_k(\ma) = 2 \|\ma\|^{\delta}_{(k)}$.
\end{proposition}
\begin{proof}
Let $\ma$ have the eigenvalue decomposition $\ma = \mq \mLambda \mq^*$, such that the eigenvalues are in nonincreasing order, and $\mq^*\mU\mq$ is the exchange matrix. Then, 
\[
\|\ma - \mU\ma\mU^*\|_{(k)} = \|\mLambda - (\mq^*\mU\mq)\mLambda(\mq^*\mU\mq)^*\|_{(k)} = \sum_{i=1}^k |\lambda_i-\lambda_{n-i+1}|^{\downarrow}.
\]
\end{proof}

A special case of inequality~\eqref{wm:generalH} states that for Hermition $\ma$ and $\mb$, we have
\begin{equation}
\label{wineq1}
\vsigma([\ma,\mb]) \prec_w \frac{|\vlambda^{\downarrow}(\ma)-\vlambda^{\uparrow}(\ma)\|\vlambda^{\downarrow}(\mb)-\vlambda^{\uparrow}(\mb)|}{2}.
\end{equation}
Before proving that this inequality is sharp (it also implies that inequality~\eqref{wm:generalH} is sharp), let us define a family of rotation matrices.
\begin{definition}
 We define the rotation matrix $\mr_n(\theta)$ for even and odd $n$ as
$$
\scalemath{0.87}{\begin{bmatrix}
\cos{\theta} & 0 &\hdots & 0 & -\sin{\theta} \\
0 & \ddots & & \adots & 0 \\
\vdots & & \begin{matrix}
\cos{\theta} & -\sin{\theta} \\
\sin{\theta} & \cos{\theta}
\end{matrix} & & \vdots\\
 0 & \adots & & \ddots & 0 \\
\sin{\theta} & 0 &\hdots & 0 & \cos{\theta}
\end{bmatrix}\ 
\text{and}\
\begin{bmatrix}
\cos{\theta} & 0 &\hdots & 0 & -\sin{\theta} \\
0 & \ddots & & \adots & 0 \\
\vdots & & \begin{matrix}
\cos{\theta} & 0 & -\sin{\theta} \\
0 & 1 & 0 \\
\sin{\theta} & 0 & \cos{\theta}
\end{matrix} & & \vdots\\
 0 & \adots & & \ddots & 0 \\
\sin{\theta} & 0 &\hdots & 0 & \cos{\theta}
\end{bmatrix}},
$$
respectively, where the angle $\theta \in [0,2\pi)$. This family of matrices can be viewed as the direct sum of two by two rotation matrices. 
\label{defR}
\end{definition}

\begin{theorem}
\label{thm1}
 Let $\ma,\mb$ be two Hermitian matrices with the eigenvalue decompositions $\ma = \mq \mLambda \mq^*$ and $\mb = \mv \mD \mv^*$. Assume that $\mLambda = \diag([\lambda_1,\cdots, \lambda_n])$ and $\mD = \diag([d_1,\cdots, d_n])$ where $\lambda_1 \geq \cdots \geq \lambda_n$ and $d_1 \geq \cdots \geq d_n$. Then, we have
\begin{equation}
  \mq\mr_n(\tfrac{\pi}{4})\mv^* \in \argmax_{\mU \in \mathbf{U}(n)} \vertiii{\,[\ma , \mU\mb\mU^*]\,}.
    \label{opt11}
\end{equation}
\end{theorem}
\begin{proof}
Note that the matrix $\tilde{\mU} := \mq\mr_n(\tfrac{\pi}{4})\mv^*$ is indeed unitary, and it is independent of the eigenvalues. At the maximum point, we have the equality
$$
\vsigma(\ma\tilde{\mU}\mb\tilde{\mU}^* - \tilde{\mU}\mb\tilde{\mU}^*\ma) = \vsigma(\mLambda\mr_n\mD\mr_n^*-\mr_n(\tfrac{\pi}{4})\mD\mr_n(\tfrac{\pi}{4})^*\mLambda).
$$
And the matrix $\mm =\mLambda\mr_n(\tfrac{\pi}{4})\mD\mr_n(\tfrac{\pi}{4})^*-\mr_n(\tfrac{\pi}{4})\mD\mr_n(\tfrac{\pi}{4})^*\mLambda$ is an anti-diagonal skew Hermitian matrix, hence $(\mm\mm^*)^{1/2}$ is diagonal with values
$$
\sigma_i(\mm) = \frac{|\lambda_i-\lambda_{n-i+1}\|d_i-d_{n-i+1}|}{2}, \qquad \text{for } i=1,\dots,n.
$$
From inequality~\eqref{wineq1}, we know that these specific singular values majorize the singular values of $[\ma, \mU\mb\mU^*]$ for an arbitrary unitary matrix $\mU$. Hence, by Fan dominance theorem $\tilde{\mU}$ is a maximizer of the above optimization problem.
\end{proof}

One can interpret Theorem~\ref{thm1} as follows: two Hermitian matrices with fixed eigenvalues are maximally non-commutative when their eigenspaces are rotated $\mr_n(\tfrac{\pi}{4})$ relative to each other. As an example, let represent the eigenvalues and eigenvectors of two $2 \times 2$ real positive definite matrices as two ellipses. This can be viewed as the linear transformation of a unit circle, when the transformation is represented by a positive definite matrix. We are allowed to rotate the matrices so that they become maximally non-commutative. The maximal configuration is illustrated in Fig.~\ref{fig:2ellipsis}.

\begin{figure}[t]
    \centering
    \includegraphics[width=0.4\textwidth]{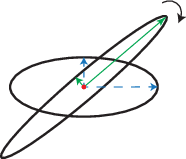}
    \caption{Two $2 \times 2$ real positive definite matrices are maximally non-commutative when their corresponding eigenspaces have the above configuration relative to each other. Here, each ellipse represents a matrix. The length of the semi-axes of the ellipses are the eigenvalues of the positive definite matrix, and the directions of the semi-axes are the corresponding eigenvectors.}
    \label{fig:2ellipsis}
\end{figure}

\subsection{Miscellaneous results} %Majorization bound for other operations with scalar invariance}
We now briefly note (without proof) some additional results; we refer the interested reader to the longer version~\cite{zadeh2021} of this work for proofs and more elaborate context motivating the noted results.

\iffalse
For $\vdelta(\ma, 0) = \min_{\alpha \in \mathbb{C}} \;\; \frac{\|\ma- \alpha \mi\|_{(k)} +k |\alpha|}{2}$, we have
\begin{equation}
    \vsigma (\ma) = 2\vdelta(\ma, 0).
\end{equation}
Since by triangle inequality we have for any $\alpha$
\[
\|\ma\|_{(k)} \leq \|\ma-\alpha\mi_n\|_{(k)}+\|\alpha\mi_n\|_{(k)},
\]
and if we let $\alpha=0$, the inequality becomes equality.

We also have $\vdelta(\ma,\ma^*)=\vdelta(\ma)$, thus $\vsigma(\ma-\ma^*) \prec_w 2\vdelta(\ma)$, which means that $|\vlambda(\ms(\ma))| \prec_w \vdelta(\ma)$ and $\vsigma(\ma^2-\ma \ma^*) \prec_w 2\vsigma(\ma)\vdelta(\ma)$. Using a similar argument we can deduce:
\fi
\begin{proposition}
For any square matrix $\mx \in \C^{n\times n}$, we have
\begin{align*}
 \sum_{i=1}^k \delta_i(\mx,\mx^*)  = \min_{\alpha \in \mathbb{R}} \|\mx+\alpha\mi_n\|_{(k)},\quad
 \sum_{i=1}^k \delta_i(-\mx,\mx^*) = \min_{\alpha \in \mathbb{R}} \|\mx+i \alpha\mi_n\|_{(k)}.  
\end{align*}
\end{proposition}
They enable us obtaining majorization bounds for $\vsigma(\ma\mb^*-\mb\ma)$ and $\vsigma(\ma\mb^*+\mb\ma)$.
% Now we want to derive similar inequalities for the commutators. First, we need the following lemma:
% Also, note that the inequality~\eqref{wm:generalH} implies majorization bounds for other operations:
% \[
% \vsigma(\ma^2\mb^2 - (\ma\mb)^2) = \vsigma(\ma[\ma,\mb]\mb) \prec_w 2\vdelta(\ma)\vdelta(\mb)\vsigma(\ma)\vsigma(\mb).
% \]
% Indeed we can solve similar optimization problems using the same technique. First, let us generalize the inequalities in the last part.
We can generalize Corollary~\ref{coroll:zero1} as follows:
\begin{proposition}
Let $\ma_1, \dots, \ma_k$ be square matrices of the same size. Then we have
$$
\vsigma\big([\ma_1, [\ma_2, \dots, [\ma_{k-1},\ma_k]\dots ]]\big) \prec_w 2^{k-1} \vsigma(\ma_1) \vdelta(\ma_2) \dots \vdelta(\ma_k).
$$
% This follows by repeated application of $\vsigma([\mx, \my]) \prec_w 2 \vdelta(\mx)\vsigma(\my)$.
%If $\ma_1$ is Hermitian, then we % in the very last inequality we can use $\vsigma([\mx, \my]) \prec_w 2 \vdelta(\mx)\vdelta(\my)$ and obtain
% $$
% \vsigma\big([\ma_1, [\ma_2, \dots, [\ma_{k-1},\ma_k]\dots ]]\big) \prec_w 2^{k-1} \vdelta(\ma_1) \vdelta(\ma_2) \dots \vdelta(\ma_k),
% $$
While if all $\ma_1, \dots, \ma_k$ are Hermitian,
$$
\vsigma\big([\ma_1, [\ma_2, \dots, [\ma_{k-1},\ma_k]\dots ]]\big) \prec_w \frac{|\vlambda^{\downarrow}(\ma_1)-\vlambda^{\uparrow}(\ma_1)|\cdots|\vlambda^{\downarrow}(\ma_k)-\vlambda^{\uparrow}(\ma_k)|}{2}.
$$
\end{proposition}
\iffalse %{
\begin{proposition}
For Hermitian matrices $\ma,\mb \in \C^{n\times n}$ with the eigenvalue decompositions $\ma = \mq \mLambda \mq^*$ and $\mb = \mv \mD \mv^*$, we have
\begin{equation*}
    \begin{split}
       \mq\mr_n(\tfrac{\pi}{4})\mv^* \in  \argmax_{\mU \in \mathbf{U}(n)} & \vertiii{[\ma,[\ma, \mU\mb\mU^*]]}  =
    \argmax_{\mU \in \mathbf{U}(n)}\vertiii{[\ma,[\ma,[\ma,[\ma,\mU\mb\mU^*]]]]} \\
    & =\argmax_{\mU \in \mathbf{U}(n)} \vertiii{[\ma,[\mU\mb\mU^*,[\mU\mb\mU^*,[\mU\mb\mU^*,\ma]]]]}
    \end{split}
\end{equation*}
\end{proposition}
\fi %}
Discrepancy values could be useful when there is a multiplicative scalar invariance.
\begin{theorem}
  \label{thm:scaling}
Let $\ma,\mb\in \C^{n\times n}$. Then,
\[
\vsigma(e^{\ma}\mb e^{-\ma}) \prec_w \vsigma(\mb) \exp{(2\vdelta(\ma))},
\]
where $\exp{(\cdot)}$ denotes the elementwise exponential.
\end{theorem}
\iffalse %{
\begin{proof}
We know that (see e.g.,~\cite{bellman1997introduction})
$$
e^{\ma}\mb e^{-\ma} = \mb +[\ma,\mb] + [\ma,[\ma,\mb]]/2!+\cdots.
$$
Therefore, $\vsigma(e^{\ma}\mb e^{-\ma}) \prec_w \vsigma(\mb) +\vsigma([\ma,\mb]) + \vsigma([\ma,[\ma,\mb]]/2!)+\cdots$, so that
\[
\vsigma(e^{\ma}\mb e^{-\ma}) \prec_w \vsigma(\mb) \sum_{j=0}^{\infty} \frac{(2\vdelta(\ma))^j}{j!}.
\]
Using the Taylor expansion of $\exp(\cdot)$ we get the result.
\end{proof}
\fi %}
%Theorem~\ref{thm:scaling} yields the following:
\begin{corollary}
Let $\ma$ be strictly positive definite, and $\mb \in \C^{n\times n}$. Then, 
\[
\vsigma(\ma\mb\ma^{-1}) \prec_w \vsigma(\mb) \frac{\vlambda^{\downarrow}(\ma)}{\vlambda^{\uparrow}(\ma)}.
\]
\end{corollary}
These inequalities help us solve some spectral optimization problems, for instance:
\begin{corollary}
 Let $\ma$ be a positive definite matrix and $\mb \in \C^{n\times n}$ the eigendecomposition $\ma = \mv \mLambda \mv^*$ and singular value decompositions $\mb = \mq \mD \mP^*$. Assume that $\mLambda = \diag([\lambda_1,\cdots, \lambda_n])$ and $\mD = \diag([d_1,\cdots, d_n])$ where $\lambda_1 \geq \cdots \geq \lambda_n$ and $d_1 \geq \cdots \geq d_n$. Then we have
\begin{equation}
 (\mv\mr_n(\tfrac{\pi}{4})\mq^*,\mv\mr_n(\tfrac{\pi}{4})\mP^*) \in\argmax_{\mU_1,\mU_2 \in \mathbf{U}(n)} \vertiii{\ma\mU_1\mb\mU^*_2\ma^{-1}}.
    \label{opt12}
\end{equation}
\end{corollary}

%%% Local Variables:
%%% mode: latex
%%% TeX-master: "main"
%%% End:

\section{Calculation of discrepancy values}
\label{sec: calculation}
As shown earlier, for Hermitian $\ma$, we know the discrepancy values in closed form: $\delta_1(\ma),\delta_2(\ma)=\frac{|\lambda_1-\lambda_n|}{2}, \delta_3(\ma),\delta_4(\ma)=\frac{|\lambda_2-\lambda_{n-1}|}{2}$, $\ldots$. We can also find the discrepancy of some other types of matrices easily. Indeed, since $\vdelta(e^{i\theta}\ma)=\vdelta(\ma)$, and using the formula for discrepancy of Hermitian matrices,  one can easily show the following simple yet intriguing result:
\begin{proposition}
For any normal matrix whose eigenvalues lie on a straight line in the complex plane and are symmetric about the origin we have $\vdelta(\ma)=\vsigma(\ma)$.
\end{proposition}
\begin{corollary}
\label{coroll:antisymm}
The singular and discrepancy values of a real skew-symmetric matrix are equal. Furthermore, if $\ma$ and $\mb$ are real symmetric matrices, we have
\[
    \vdelta([\ma,\mb])=\vsigma([\ma,\mb]).
\]
\end{corollary}
%In the next parts, we talk about computing such values for other classes of matrices.
\subsection{Computing discrepancy for normal matrices}
We know that $\vdelta(\ma^*)=\vdelta(\ma)$ and $\vdelta(\mU^*\ma\mU)=\vdelta(\ma)$ for any unitary matrix $\mU$; hence, if $\ma$ is a normal matrix with the eigenvalue decomposition $\ma = \mU^* \mLambda \mU$ then $\vdelta(\ma) = \vdelta(\mLambda)$. Consequently, 
\begin{equation}
  \label{eq:12}
    \sum_{i=1}^k \delta_i(\ma) = \min_{\alpha \in \mathbb{C}}\; \max_{1\leq i_{1}<i_{2}<i_{k}\leq n}\; \sum_{j=1}^k |\lambda_{i_{j}}-\alpha|.
\end{equation}
Problem~\eqref{eq:12} can be reformulated as the second-order cone programming problem:
\begin{equation}
  \label{eq:13}
    \begin{split}
        \sum_{i=1}^k \delta_i(\ma) = \min_{\alpha \in \mathbb{C}, q \in \mathbb{R}, \vu, \vx \in \mathbb{R}^n} \quad & \mathbf{1}_n^T\vu+k q, \\
        \text{subject to }\quad & \vu \geq \vx - q\mathbf{1}_n,\quad \vx \geq |\mlambda - \alpha\mathbf{1}_n|, \quad \vu \geq 0,
    \end{split}
\end{equation}
where $|\cdot|$ denotes the element-wise modulus of a complex vector. The optimal objective function value of~\eqref{eq:13} equals $\sum_{i=1}^k \delta_i(\ma)$. %Moreover, for some special cases such as Hermitian and skew-Hermitian matrices, the optimization problem has a closed form solution. 

Notice that for any normal matrix $\ma$ we have $\delta_1(\ma)=\delta_2(\ma)=R$, where $R$ corresponds to the radius of the smallest circle containing all the eigenvalues; see Fig.~\ref{fig:mesh1}. Also, we can see that the optimum $\alpha$ which minimizes $\|\ma-\alpha\mi_n\|_{(n)}$ is just the geometric median of the eigenvalues in the two-dimensional complex plane.
\begin{figure}[t]
    \centering
    \includegraphics[width=0.4\textwidth]{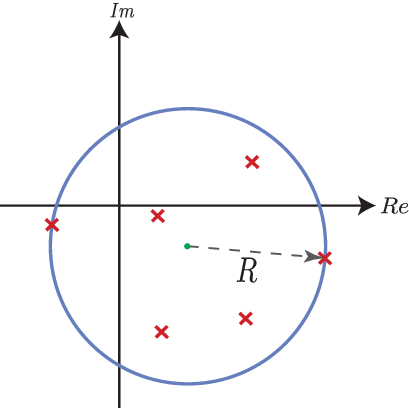}
    \caption{The minimal enclosing circle of the eigenvalues}
    \label{fig:mesh1}
    \vspace{-5pt}
\end{figure}

\vspace*{-5pt}
\subsection{SDP representation for discrepancy seminorms}
From Theorem~\ref{lemma_min}, we know that $\|\ma\|_{(k)}^{\delta} = \min_{\alpha\in\mathbb{C}} \|\ma-\alpha\mi_n\|_{(k)}$, where the latter is a convex function in $\alpha$. Hence, we can find $\vdelta(\ma)$ by employing a series of nonsmooth convex optimization methods. We propose below a more convenient way to compute the discrepancy values via semidefinite programming (SDP). To that end, recall that
\begin{equation}
  \label{eq:7}
    \sum_{i=1}^n \delta_i(\ma) = \max_{\substack{\mU^*\mU = \mi_n \\ \mv^*\mv = \mi_n \\ \trace(\mU^*\mv)=0}} \R \trace(\ma \mv \mU^*).
\end{equation}
If we let $\mm = \scriptsize\begin{bmatrix}
\mU^*\\
\mv^*
\end{bmatrix}\begin{bmatrix}
\mU & \mv
\end{bmatrix}$ and $\mathcal{H}(\ma)=\scriptsize\begin{bmatrix}0 & \ma\\ \ma^* &0\end{bmatrix}$, then using~\eqref{eq:7} we can write
\begin{equation}
    \sum_{i=1}^n \delta_i(\ma) = \max_{\substack{\mm = {\scriptsize\begin{bmatrix}
    \mi_n & \mk\\
    \mk^* & \mi_n
    \end{bmatrix}} \succeq 0\\ \trace(\mk)=0}} \tfrac{1}{2} \R \trace\bigl(\mathcal{H}(\ma)\mm\bigr).
\end{equation}

To derive SDP formulations for other $\sum_{j=1}^k \delta_j(\ma)$, we use the fact that for any Hermitian matrix $\mx$ with eigenvalues $\lambda_1(\mx) \geq \dots \geq \lambda_n(\mx)$ we have (see~\cite{bentalNem2001})
\[
    t \geq \lambda_1(\mx) + \dots +\lambda_k(\mx) \iff \exists (\mz, s): \mz \succeq 0, \mz+s\mi_n \succeq \mx, t \geq \R \big( \trace(\mz) + k s\big),
\]
where $t$ is a non-negative real number, and $s \in \mathbb{C}$ and $\mz$ can be complex valued in general. As a result, we have the following SDP formulation for the Ky-Fan norm:
\begin{equation*}
  \begin{split}
    \|\ma\|_{(k)}    &= \min_{\mz \succeq 0, t \geq 0, s \in \mathbb{C}} t\\
    \mathcal{H}(\ma) &\preceq \mz+s\mi_{2n},\quad t \geq \R(\trace(\mz)+ks),
  \end{split}
\end{equation*}
where $k=1,\dots,n$ and $\mz$ is a $2n \times 2n$ complex matrix. We similarly also obtain
\begin{equation}
    \label{SDP_general}
    \|\ma\|_{(k)}^\delta = \min_{\substack{\mz \succeq 0, t \geq 0, s,\alpha \in \mathbb{C} \\ {\scriptsize\begin{bmatrix}
    0 & \ma-\alpha\mi_n\\
    \ma^*-\bar{\alpha}\mi_n & 0
    \end{bmatrix}} \preceq \mz+s\mi_{2n} \\ t \geq \R(\trace(\mz)+ks)}} t,
\end{equation}
where $k=1,\dots,n$ and $\mz$ is a $2n \times 2n$ complex matrix. Note that if matrix $\ma$ is real, we can let $\mz$ and other optimization variables be real.\footnote{An implementation of this formulation is available at \href{https://github.com/PouriaZ/Discrepancy}{\it https://github.com/PouriaZ/Discrepancy}}

%%% Local Variables:
%%% mode: latex
%%% TeX-master: "main"
%%% End:

\section{Discussion and extensions}

Most of the results proven in the previous sections can be extended to linear operators of the form $\ma + \gamma \mi$, where $\ma$ is a compact operator on Hilbert spaces, and $\gamma \in \mathbb{C}$. For any compact operator $\ma:H \to K$ between Hilbert spaces $H$ and $K$, we know that $\ma$ always has countably many non-negative singular values, among which $\sigma=0$ is the only possible point of accumulation. 
For any finite $k$, the operator norm $\|\ma\|_{(k)}$ is finite, hence $\delta_k(\ma + \gamma \mi)$ is well-defined. Furthermore, if the rank of the operator is not finite, we have $\delta_k(\ma + \gamma \mi) \to 0$ as $k\to \infty$; regardless of the fact that $\ma$ belongs to the trace-class or not. 
This follows by the fact that $\|\ma + \gamma \mi\|^{\delta}_{(k)} \leq \|\ma\|_{(k)}$ for any finite $k\in \mathbb{N}$ and $\sigma_k \to 0$ as $k\to \infty$.

We end the paper with a conjecture on discrepancy values.
\begin{conj}
    \label{conj:squareM}
    For the $n \times n$ square matrices $\ma$ and $\mb$ we have
    \begin{equation}
        \vsigma ([\ma,\mb]) \prec_w 2 \vdelta(\ma)\vdelta(\mb),
    \end{equation}
    where $\vdelta(\ma)\vdelta(\mb)$ is the entrywise product of the vectors of the discrepancy values.
\end{conj}
%Using Lemma~\ref{max_delta}, this conjecture can be equivalently stated that for any unitary matrix $\mU$, we have
%$$
%    \vdelta(\mU[\ma,\mb]) \prec_w 2 \vdelta(\ma)\vdelta(\mb).
%$$
For the $2 \times 2$ matrix $\ma=[A_{ij}]$, for $i=1,2$, we can show that
\[
    \delta_i(\ma) = \tfrac{1}{\sqrt{2}}\Big[\tfrac{1}{2}(A_{11}-A_{22})^2+A_{12}^2+A_{21}^2 \pm |A_{12}-A_{21}|\sqrt{(A_{12}+A_{21})^2+(A_{11}-A_{22})^2}\Big]^{1/2}.
\]
By inspection, one can verify that $\vdelta(\ma) = \vsigma\big(\ma-\frac{A_{11}+A_{22}}{2}\mi_2\big)$. Thus we have
\[
    \vsigma([\ma,\mb]) = \vsigma\Big(\big[\ma-\tfrac{A_{11}+A_{22}}{2}\mi_2, \mb-\tfrac{B_{11}+B_{22}}{2}\mi_2\big]\Big) \prec_w 2 \vdelta(\ma)\vdelta(\mb).
\]
Therefore the conjecture is true for any two by two matrices. Furthermore, the conjecture has been proven for some general cases in this paper (see Remark~\ref{numb} and Corollary~\ref{generalMajIneq}). However, the conjecture in its full generality is an open problem.

%%% Local Variables:
%%% mode: latex
%%% TeX-master: "main"
%%% End:

% \section{Acknowledgement}
% This work has been funded by 
% \textsc{Nsf-Career} (IIS-1846088), \textsc{Nsf-Tripods+X} (DMS-1839258); EECS MathWorks Engineering Fellowship.

{
  \bibliographystyle{plainnat}
  \bibliography{egbib}
}
\pagebreak
\appendix
%} 

\section{The X decomposition for Hermitian matrices}
\label{apendix:Xmatrix}

In this part, we propose a decomposition for the Hermitian matrices based on their discrepancy values and vectors. First, let us introduce a family of matrices.

\begin{definition}[X matrices]
We call a square matrix which is allowed to have nonzero entries only on the main diagonal, or on the anti-diagonal, an X matrix (or star matrix).
\end{definition}
% \begin{proposition}
% The X matrix can be viewed as a bunch of 2 by 2 matrices, and the coupling is centrosymmetric. Using a unitary similarity transformation, we can turn the matrix to be like a block diagonal matrix, where each block is a two by two matrix. The eigenvalues of an X matrix is
% \[
% \lambda_i = \frac{a_i+a_{n-i+1}}{2}\pm \sqrt{\Big(\frac{a_i-a_{n-i+1}}{2}\Big)^2+b_i b_{n-i+1}}, \qquad i=1,\dots, \floor{\tfrac{n}{2}}
% \]
% where $a_i$ are the entries on the main diagonal and $b_i$ are the entries on the anti-diagonal. If $n$ is odd, we have another eigenvalue $\lambda_{\ceil{\tfrac{n}{2}}} = a_{\ceil{\tfrac{n}{2}}}$, which is the element at the center of the matrix. 
% \end{proposition}

\begin{remark}
The set of X matrices are unitarily similar to matrices of the form $B_1 \oplus B_2 \oplus \cdots \oplus B_n$, where $B_i$ are two by two matrices, plus one scalar if it has odd dimensionality.
% \[
% B_i = \begin{bmatrix}
% a_i & b_i \\
% c_i & d_i
% \end{bmatrix},
% \]
% or $B_i = \begin{bmatrix}
% a_i
% \end{bmatrix}$.
\end{remark}

% \begin{remark}
% The set of X matrices form an Algebra over $\mathbb{R}$ or $\mathbb{C}$. Also, the commutator of two X matrices is another X matrix.
% \end{remark}
% \begin{definition}
% A square matrix is called a skew-centrosymmetric matrix if $\ma_{i,j}=-\bar{\ma}_{n-i+1,n-j+1}$. 
% % More specifically, $\ma_{2n\times 2n}$ with the following form
% % \[
% % \ma = \begin{bmatrix} L_{n\times n} & R_{n\times n} \\ -R^*_{n\times n} & -L^*_{n\times n}\end{bmatrix},
% % \]
% % or the matrix $\ma_{(2n+1)\times (2n+1)}$ with the following form
% % \[
% % \ma = \begin{bmatrix} L_{n\times n} & q_{n\times 1} & R_{n\times n} \\ p_{1\times n} & m_{1\times 1} & -\bar{p}_{1\times n} \\ -R^*_{n\times n} & -\bar{q}_{n\times 1} & -L^*_{n\times n}\end{bmatrix},
% % \]
% % is a skew-centrosymmetric matrix. Note that $\bar{q}$ denotes the operator of elementwise conjugation on the vector $q$.
% \end{definition}
% \begin{notation}
Let the notations $CX(a_1, \dots, a_n)(b_1, \dots, b_n)$ and $CX(a_1, \dots, a_n)(c)(b_1, \dots, b_n)$ denote the following centrosymmetric X matrices, respectively:
$$
\scalemath{0.9}{\begin{bmatrix}
a_1 & 0 &\hdots & 0 & \bar{b}_1 \\
0 & \ddots & & \adots & 0 \\
\vdots & & \begin{matrix}
a_n & \bar{b}_n \\
b_n & \bar{a}_n
\end{matrix} & & \vdots\\
 0 & \adots & & \ddots & 0 \\
b_1 & 0 &\hdots & 0 & \bar{a}_1
\end{bmatrix}, 
\qquad
\begin{bmatrix}
a_1 & 0 &\hdots & 0 & \bar{b}_1 \\
0 & \ddots & & \adots & 0 \\
\vdots & & \begin{matrix}
a_n & 0 & \bar{b}_n \\
0 & c & 0 \\
b_n & 0 & \bar{a}_n
\end{matrix} & & \vdots\\
 0 & \adots & & \ddots & 0 \\
b_1 & 0 &\hdots & 0 & \bar{a}_1
\end{bmatrix}}.
$$
% sd
% $$
% \begin{bmatrix}
% a_1 & 0 &\hdots & 0 & b_1 \\
% 0 & \ddots & & \adots & 0 \\
% \vdots & & \begin{matrix}
% a_n & b_n \\
% -\bar{b}_n & -\bar{a}_n
% \end{matrix} & & \vdots\\
%  0 & \adots & & \ddots & 0 \\
% -\bar{b}_1 & 0 &\hdots & 0 & -\bar{a}_1
% \end{bmatrix}
% $$
% and $SCX(a_1, \dots, a_n)(c)(b_1, \dots, b_n)$ denotes
% $$
% \begin{bmatrix}
% a_1 & 0 &\hdots & 0 & b_1 \\
% 0 & \ddots & & \adots & 0 \\
% \vdots & & \begin{matrix}
% a_n & 0 & b_n \\
% 0 & c & 0 \\
% -\bar{b}_n & 0 & -\bar{a}_n
% \end{matrix} & & \vdots\\
%  0 & \adots & & \ddots & 0 \\
% -\bar{b}_1 & 0 &\hdots & 0 & -\bar{a}_1
% \end{bmatrix}.
% $$
% Moreover, let $CX(a_1, \dots, a_n)(b_1, \dots, b_n)$ and $CX(a_1, \dots, a_n)\circled{c}(b_1, \dots, b_n)$ denote the centrosymmetric version of the above matrices; i.e. the matrix $\ma$ in $CX$ has the property that $\ma_{i,j}=\bar{\ma}_{n-i+1,n-j+1}$.
% \end{notation}

% \begin{definition}[Centrosymmetric matrix]
% A square matrix is centrosymmetric if . Similarly, we denote the centrosymmetric X matrix by $CX(a_1, \dots, a_n)(c)(b_1, \dots, b_n)$
% \end{definition}

\begin{proposition}[X decomposition]
Any $n \times n$ Hermitian matrix $\ma$ can be decomposed as $\ma = \mU \mx \mv^*$, where $\mx$ is a centrosymmetric X matrix; moreover, $\mU$ and $\mv$ are unitary matrices such that $\mU^*\mv=\mj_n$, the exchange matrix.
%(Also, $\sqrt{2}\mU+\sqrt{2}\mv$ is a unitary matrix).
\end{proposition}
\begin{proof}
Let $\mr_n(\cdot)$ be the $n \times n$ orthogonal matrix defined in~\ref{defR}, then using the eigenvalue decomposition of $\ma$ we have
\[
\ma = \mq \mLambda \mq^* = \mq \mr_n(\tfrac{\pi}{4}) (\mr_n(\tfrac{\pi}{4})^* \mLambda \mj_n \mr_n(\tfrac{\pi}{4})^*) (\mq\mj_n\mr_n(\tfrac{\pi}{4})^*)^*.
\]
Now let $\mU = \mq \mr_n(\tfrac{\pi}{4})$, $\mv = \mq\mj_n\mr_n(\tfrac{\pi}{4})^*$, and $\mx = \mr_n(\tfrac{\pi}{4})^* \mLambda \mj_n \mr_n(\tfrac{\pi}{4})^*$. One can easily verify all the conditions in the proposition are satisfied.
\end{proof}

% \begin{proposition}
In the previous proposition, we can show that the columns of $\mU$ and $\mv$ are vectors $\vx_i$ and $\vy_i$ in the definition of $\delta$ in~\ref{main_def}. Also if $n$ is an even number we have 
\[
\mx = CX(\delta_1(\ma),\delta_3(\ma),\dots)(\alpha_1^*(\ma),\alpha_3^*(\ma), \dots),
\]
and otherwise we have
\[
\mx = CX(\delta_1(\ma),\delta_3(\ma),\dots)(\alpha_n^*)(\alpha_1^*(\ma),\alpha_3^*(\ma), \dots),
\]
where $\delta_{i}(\ma)=(\lambda_{i}-\lambda_{n-i+1})/2$ and $\alpha^*_{i}(\ma)=(\lambda_{i}+\lambda_{n-i+1})/2$.

\end{document}